\documentclass{siamart190516}

\usepackage[a4paper, left=4cm, right=4cm, bottom=3cm]{geometry}
\usepackage{amssymb,amsmath,amsfonts,mathrsfs,mathabx,dsfont,hyperref,graphicx,tikz-cd, xcolor, soul}
\usepackage{algorithmic}

\usepackage[sort]{cite}

\author{Nick Dewaele\thanks{KU Leuven, Department of Computer Science, Celestijnenlaan 200A, B-3001 Leuven, Belgium. \email{nick.dewaele@kuleuven.be}} \and Paul Breiding\thanks{MPI MiS, Inselstr.\ 22, 04103 Leipzig, Germany. \email{paul.breiding@mis.mpg.de}. P.\ B.\ is funded by the Deutsche Forschungsgemeinschaft (DFG) -- Projektnummer 445466444.} \and Nick Vannieuwenhoven\thanks{KU Leuven, Department of Computer Science, Celestijnenlaan 200A, B-3001 Leuven, Belgium. \email{nick.vannieuwenhoven@kuleuven.be}. N.\ V.\ was supported by a Postdoctoral Fellowship of the Research Foundation---Flanders (FWO) with project 12E8119N.}}
\title{The condition number of many tensor decompositions is invariant under Tucker compression}
\date{}

\DeclareMathOperator{\Span}{span}

\newtheorem{remark}[theorem]{Remark}

\crefname{equation}{}{}
\crefname{equation}{}{}
\crefname{figure}{Figure}{Figures}
\crefname{section}{Section}{Sections}
\crefname{table}{Table}{Tables}
\crefname{lemma}{Lemma}{Lemmata}
\crefname{appendix}{Appendix}{Appendix}
\crefname{prop}{Proposition}{Propositions}
\crefname{thm}{Theorem}{Theorems}
\crefname{cor}{Corollary}{Corollaries}
\crefname{dfn}{Definition}{Definitions}
\crefname{hyp}{Hypothesis}{Hypotheses}
\crefname{notation}{Notations}{Notations}
\crefname{rem}{Remark}{Remarks}
\crefname{claim}{Claim}{claims}
\crefname{assumption}{Assumption}{Assumptions}

\newcommand{\R}{\mathbb{R}}

\DeclareMathAlphabet{\mathpzc}{OT1}{pzc}{m}{it}

\newcommand{\norm}[1]{\left\|#1\right\|}
\newcommand{\vb}[1]{\mathbf{#1}}

\begin{document}
\maketitle

\begin{abstract}
We characterise the sensitivity of several additive tensor decompositions with respect to perturbations of the original tensor.
These decompositions include canonical polyadic decompositions, block term decompositions, and sums of tree tensor networks. 
Our main result shows that the condition number of all these decompositions is invariant under Tucker compression. This result can dramatically speed up the computation of the condition number in practical applications. We give the example of an $265\times 371\times 7$ tensor of rank $3$ from a food science application whose condition number was computed in $6.9$ milliseconds by exploiting our new theorem, representing a speedup of four orders of magnitude over the previous state of the art.
\end{abstract}


\section{Introduction}

A tensor of order $D$ is a $D$-array of size $n_1 \times \dots \times n_D$. In numerous applications (see, e.g., \cite{Kolda2009,Sidiropoulos2017} and the references therein), one seeks a decomposition that expresses a tensor $\mathpzc{A}$ as a sum of $R$ elementary terms:
\begin{equation}
    \label{eq:decomposition}
\mathpzc{A} = \mathpzc{A}_1 + \dots + \mathpzc{A}_R,
\end{equation}
where $\mathpzc{A}_r\in\mathcal M_r$ and $\mathcal M_r$ is a low-dimensional \textit{manifold} in the space of tensors. Such a decomposition was called a \emph{join decomposition} in \cite{Breiding2018a}.

In this paper, we study the sensitivity properties of a certain subclass of join decompositions related to tensors. We call them \textit{structured block term decompositions} (SBTD). The formal definition of this class is given in \cref{sec:sbtd} below. Informally, an SBTD involves manifolds $\mathcal M_r$ that are defined by imposing certain (manifold) structures on the core tensor of a Tucker decomposition with fixed multilinear rank $(l_1,\dots,l_D)$.
Many commonly used decompositions are SBTDs; for instance,
\begin{itemize}
    \item[$\circ$] sums of rank-$1$ tensors, i.e., canonical polyadic decomposition (CPD) \cite{Hitchcock1927},
    \item[$\circ$] sums of Tucker decompositions, i.e., block term decomposition (BTD) \cite{DeLathauwer2008},
    \item[$\circ$] sums of tensor train decompositions \cite{Oseledets2011, Ehrlacher2021}, and
    \item[$\circ$] sums of hierarchical Tucker decompositions \cite{Hackbusch2009,Grasedyck2010}.
\end{itemize}
The \textit{condition number} is one way to measure the sensitivity of a tensor decomposition relative to perturbations of the tensor. For general join decompositions this number was analysed in \cite{Breiding2018a}.
One main result we establish in this paper is that the condition number of an SBTD is invariant under \textit{Tucker compression}.

Recall that a \textit{Tucker decomposition} \cite{Tucker1966} represents $\mathpzc{A}$ in a tensor product subspace by expressing
\[
\mathpzc{A} = (Q_1, \dots, Q_D) \cdot \mathpzc{G}
:= \sum_{i_1=1}^{m_1} \dots \sum_{i_D=1}^{m_D} g_{i_1,\dots,i_D} \vb{q}_{i_1}^{1} \otimes \dots \otimes \vb{q}_{i_D}^D,
\]
where $Q_d = [ \vb{q}_i^d ]_i \in \mathbb{R}^{n_d \times m_d}_\star$ with $n_d \ge m_d$ have linearly independent columns for each $d = 1,\dots,D$. 

The \textit{core tensor} $\mathpzc{G}$ is often much smaller than $\mathpzc{A}$, and it gives the coordinates of~$\mathpzc{A}$ with respect to the tensor product basis $Q_1 \otimes \dots \otimes Q_D$. Note that we will switch freely between two equivalent notations for Tucker decomposition: the first, $(Q_1, \dots, Q_D)\cdot\mathpzc{G}$, is a common notation \cite{DeSilva2008} for \textit{multilinear multiplication}, while the second, $(Q_1 \otimes \dots \otimes Q_D)\mathpzc{G}$ emphasises that a Tucker decomposition consists of taking a particular linear combination of the tensors in a tensor product basis $Q_1 \otimes \dots \otimes Q_D$. Herein, $Q_1 \otimes \dots \otimes Q_D$ denotes the tensor product of matrices, which acts linearly on rank-$1$ tensors by $(Q_1, \dots, Q_D) \cdot (v_1\otimes \cdots\otimes v_D) := (Q_1v_1\otimes \cdots\otimes Q_Dv_D)$.
In coordinates, this matrix is given by $Q_1 \otimes \dots \otimes Q_D = [ \vb{q}_{i_1}^1 \otimes \dots \otimes \vb{q}_{i_D}^D ]_{i_1,\dots,i_D}$; see \cite{Greub1978}.

Originally proposed for CPD, Tucker compression \cite{Bro1998} consists of expressing a tensor $\mathpzc{A} \in \mathbb{R}^{n_1 \times \dots \times n_D}$ in coordinates in the smallest tensor product subspace in which it lives, in order to speed up the computation of decompositions of the form \cref{eq:decomposition}. That is, before computing the CPD, one first computes a Tucker decomposition, expressing $\mathpzc{A}=(Q_1,\dots,Q_D)\cdot\mathpzc{G}$. Then, one computes the CPD of the core tensor $\mathpzc{G}$. The obtained decomposition can then be extended to a decomposition of the original tensor $\mathpzc{A}$ by multilinear multiplication with the basis $(Q_1, \dots, Q_D)$.
Since there are efficient algorithms \cite{DeLathauwer2000,Vannieuwenhoven2012} for computing an approximate Tucker decomposition of~$\mathpzc{A}$, contrary to the mostly optimization-based algorithms for computing CPDs, this compress--decompose--decompress approach can often reduce the overall computation time \cite{Bro1998}.
Another main result of this paper characterises which manifolds~$\mathcal{M}_r$ in the join decomposition \cref{eq:decomposition} are compatible with Tucker compression.

The topic of this article is to characterise how a decomposition of the form \cref{eq:decomposition} changes if $\mathpzc{A}$ is corrupted by noise. In order to obtain a robust interpretation the elementary terms, it is essential to quantify how sensitive they are to the perturbations. As explained in \cite{Breiding2018a}, under certain mild conditions,
$\mathpzc{A}$ has an isolated decomposition  $a = (\mathpzc{A}_1,\dots,\mathpzc{A}_R)$ and we can find a local inverse function $\Sigma^{-1}_a$ of the \emph{addition map}
$\Sigma: \mathcal M_1\times \dots\times \mathcal M_R\to \mathbb R^{n_1\times \dots\times n_D},\;(\mathpzc{A}_1, \dots, \mathpzc{A}_R) \mapsto \mathpzc{A}_1+\dots+\mathpzc{A}_R$. The sensitivity of the elementary terms $\mathpzc{A}_r$ can be measured by the condition number \cite{Rice1966}
\begin{equation}
    \label{eq:defcond}
\kappa^{\mathrm{SBTD}}(\mathpzc{A}_1,\dots,\mathpzc{A}_R) :=
\lim_{\delta \rightarrow 0}\,
\sup_{\substack{\widetilde{\mathpzc{A}} \in \mathcal{I}:\norm{\mathpzc{A} - \widetilde{\mathpzc{A}}} \le \delta}} \,\frac{\norm{\Sigma^{-1}_a(\mathpzc{A}) - \Sigma^{-1}_a (\widetilde{\mathpzc{A}})}}{\norm{\mathpzc{A} - \widetilde{\mathpzc{A}}}},
\end{equation}
where $\mathcal{I}$ is the set of valid perturbations (more on this below), and $\norm{\cdot}$ denotes both the Euclidean norm on the ambient space $\R^{n_1\times\dots\times n_D}$ and the product Euclidean norm on $\R^{n_1\times\dots\times n_D} \times \dots \times \R^{n_1\times\dots\times n_D}$. The condition number measures $a=(\mathpzc{A}_1,\dots,\mathpzc{A}_R)$ in one piece as tensors. It does not measure how the points are \textit{parametrised}, which would introduce a number of complications.\footnote{See \cite{V2017} for how to deal with such complications in the context of the CPD.}
Furthermore, a priori, the condition number depends on both input $\mathpzc{A}$ and output $(\mathpzc{A}_1,\dots,\mathpzc{A}_R)$ because it is defined in terms of a \emph{local} inverse \cite{Burgisser2013}. However, since $\mathpzc{A}$ depends uniquely on $(\mathpzc{A}_1,\dots,\mathpzc{A}_R)$ we can write the condition number as a function of the output only.
We have 
\begin{equation}
    \label{eq:fstOrderErrBound}
\norm{\Sigma^{-1}_a(\mathpzc{A}) - \Sigma^{-1}_a(\widetilde{\mathpzc{A}})}
\le \kappa^{\mathrm{SBTD}}(\mathpzc{A}_1,\dots,\mathpzc{A}_R)
\norm{\mathpzc{A} - \widetilde{\mathpzc{A}}} + o\left(\norm{\mathpzc{A} - \widetilde{\mathpzc{A}}}\right)
\end{equation}
as an asymptotically sharp first-order error bound.
\Cref{eq:defcond} requires specifying the domain $\mathcal{I}$, which means fixing the space in which the perturbations $\widetilde{\mathpzc{A}}$ are allowed to live. There are four increasingly restrictive ways of looking at the problem:
\begin{enumerate}
    \item $\widetilde{\mathpzc{A}}\in \mathbb R^{n_1\times\dots\times n_D}$ is arbitrary and the SBTD of $\widetilde{\mathpzc{A}}$ is interpreted as the least-square minimiser $\mathrm{argmin}_{(\mathpzc{A}_1,\dots,\mathpzc{A}_R)\in\mathcal M_1\times \dots\times \mathcal M_R} \tfrac{1}{2}\norm{\widetilde{\mathpzc{A}} - (\mathpzc{A}_1+\dots+\mathpzc{A}_R)}^2$.
    \item $\widetilde{\mathpzc{A}}$ has an SBTD.
    \item $\widetilde{\mathpzc{A}}$ can be Tucker compressed to a core $\widetilde{\mathpzc{G}} \in \mathbb{R}^{m_1 \times \dots \times m_D}$ and $\widetilde{\mathpzc{G}}$ has an SBTD.
    \item $\widetilde{\mathpzc{A}}$ lives in the same tensor subspace as $\mathpzc{A}$, i.e., we have ${\mathpzc{A}} = (Q_1, \dots, Q_D) \cdot {\mathpzc{G}}$ and $\widetilde{\mathpzc{A}} = (Q_1, \dots, Q_D) \cdot \widetilde{\mathpzc{G}}$, and the cores ${\mathpzc{G}}$ and $\widetilde{\mathpzc{G}}$ both have an SBTD.
\end{enumerate}
A priori, one should expect the problem to become easier in the more restrictive cases in the sense that the condition number decreases. Indeed, the set of allowed perturbations gets strictly smaller. However, we prove the following surprising result.

\begin{theorem}
    \label{thm:informalSBTDcondInvariance}
Let $\mathpzc{A} = \mathpzc{A}_1 + \dots + \mathpzc{A}_R$ be an SBTD. The condition number $\kappa^{\mathrm{SBTD}}(\mathpzc{A}_1,\dots,\mathpzc{A}_R)$ is the same for all four domains outlined above.
\end{theorem}
This theorem is implied by \cref{thm:SBTDcondInvariance,cor:SBTDcondInvariance} below.

\Cref{thm:informalSBTDcondInvariance} is in stark contrast to some other problems in which the condition number depends on the domain. For instance,
the condition number of the matrix logarithm for perturbations constrained in the symplectic group was studied in~\cite{Arslan2019}. It was shown that the ratio between the unconstrained and constrained condition number can become arbitrarily large.

Our result indicates that computing the SBTDs of $\mathpzc{A}$ and $\mathpzc{G}$ are equally difficult from a numerical point of view. Indeed, it is known that the convergence rate of iterative methods to compute the decomposition is related to the condition number \cite{Nocedal2006, Breiding2018b,Absil2008}.
This suggests that compression, surprisingly, will not improve the local rate of convergence, even though the search space can be much smaller. Compression can, nevertheless, reduce the overall computation time when $\mathpzc{A}$ is highly compressible \cite{Bro1998}.

A major practical advantage of \cref{thm:informalSBTDcondInvariance} is that the condition number can be computed more efficiently by considering $\mathpzc{A}$ as a point in a tensor product subspace: It suffices to compute the condition number of the core $\mathpzc{G}$. An example illustrates the above significant computational advantage. Consider a rank-3 tensor of dimensions $265 \times 371 \times 7$, as in the sugar data set of \cite{Bro1998}.
Its CPD can be compressed to that of a  $3 \times 3 \times 3$-tensor.
We implemented two algorithms to compute the condition number in Julia v1.6.\ \cite{BEKV}: the one from \cite{Breiding2018a} and one based on \cref{thm:informalSBTDcondInvariance}. On a system with an Intel Xeon CPU E5-2697 v3 running on $8$ cores and $126$GB memory, this took 110 seconds and 6.9 milliseconds, respectively, which gives a speedup of over $15\,000$. If the CPD is already in compressed form, the time can be reduced further to only $0.089$ milliseconds, representing a speedup of more than a million over the state of the art.

\subsection{Outline}
We introduce the SBTD in \cref{sec:sbtd}. In \cref{sec:geometry}, we derive the geometric foundations of structured Tucker decompositions upon which the theory of the condition number is based. An algorithm to compute the condition number is outlined in \cref{sec:computeCond}. We also present qualitative properties of well or ill-conditioned SBTDs.
\Cref{sec:basicDefs} introduces subspace-constrained SBTDs and proves the main result, \cref{thm:informalSBTDcondInvariance}, which states that the condition number of SBTDs is invariant under Tucker compression.
Concluding numerical experiments are found in \cref{sec:experiments}.

\subsection{Notation}
The only norm used in this paper are the Euclidean (or Frobenius) norms for tensors and vectors, all consistently denoted by $\|\cdot\|$.
The manifold of real $n \times m$ matrices of rank $m$ is denoted as $\mathbb{R}^{n \times m}_\star$, where $n \ge m$. The $n$-dimensional sphere is $\mathbb{S}^n$. The $j$th column of the identity $\mathds{1}_n$ is $\vb{e}^{(n)}_j$. If the ambient dimension is clear from the context, we also abbreviate $\vb{e}_j:=\vb{e}^{(n)}_j$. The $d$th unfolding of a tensor $\mathpzc{A}$ is $\mathpzc{A}_{(d)}$.
For any matrix $X$ and any set of matrices $A_n$, $n = 1,\dots, N$, and any $k = 1,\dots, N+1$, we denote
$
X \otimes_k (A_1 \otimes \dots \otimes A_N) := A_1 \otimes \dots \otimes A_{k-1} \otimes X \otimes A_{k} \otimes \dots \otimes A_N
.$
For a group $G$ acting on a set $\mathcal{M}$, the projection of $x \in \mathcal{M}$ onto its $G$-orbit is $[x]_{G}$.

\section{The structured block term decomposition}
\label{sec:sbtd}
In this section, we give a formal definition of the SBTD, the main tensor decomposition that we study in this paper. Just as a BTD is a linear combination of Tucker decompositions, an SBTD is a linear combination of \emph{structured Tucker decompositions}. The structure we consider is imposed on the core tensor of the Tucker decomposition.

\begin{definition}[Tucker core structure] \label{def_Tuckercore}
A smooth submanifold $\mathcal{M} \subseteq \R^{l_1 \times \dots \times l_D}$ is a \emph{Tucker core structure} if it is a $(\mathrm{GL}(l_1)\times\dots\times\mathrm{GL}(l_D))$-homogeneous manifold:
\begin{enumerate}
    \item $\mathpzc{C} \in \mathcal{M}$ has multilinear rank equal to $(l_1,\dots,l_D)$, and
    \item $\left( A_1, \dots, A_D \right) \cdot \mathpzc{C} \in \mathcal{M}$ for all $A_d \in \mathrm{GL}(l_d)$ with $d=1,\dots,D$.
\end{enumerate}
\end{definition}

Next, we can define the $\mathcal{M}$-structured Tucker decomposition.

\begin{definition}[Structured Tucker decomposition] \label{def:GTucker}
Let $\mathcal{M} \subseteq \R^{l_1 \times \dots \times l_D}$ be a Tucker core structure.
An \emph{$\mathcal{M}$-struct\-ur\-ed Tucker decomposition} of $\mathpzc{A} \in \R^{n_1 \times \dots \times n_D}$ is an expression of the form
\begin{equation*}
\mathpzc{A} = (U_1, \dots, U_D) \cdot \mathpzc{C} = (U_1 \otimes \dots \otimes U_D) \mathpzc{C}
\end{equation*}
with $\mathpzc{C} \in \mathcal{M}$ and all $U_d \in \mathbb{R}^{n_d \times l_d}_\star$ for $d=1,\dots,D$.
\end{definition}

The first basic result we establish in \cref{sec:geometry} below ensures that the results from \cite{Breiding2018a} can be applied to study the condition number.

\begin{proposition}\label{prop_itsamanifold}
The set of all tensors $\mathpzc{A}$ admitting an $\mathcal{M}$-structured Tucker decomposition forms a smooth embedded submanifold $\mathcal{M}^{n_1,\dots,n_D} \subseteq \mathbb{R}^{n_1 \times \dots \times n_D}$, called the $\mathcal{M}$-structured Tucker manifold.
\end{proposition}

An important subclass of structured Tucker manifolds in applications are defined by \textit{tensor networks} in which the graph is a tree \cite{Orus2014}. This includes tensors with a fixed rank Tucker decomposition \cite{Tucker1966}, fixed-rank tensor train decomposition \cite{Oseledets2011}, and fixed rank hierarchical Tucker decomposition \cite{Hackbusch2009, Grasedyck2010}.

The set of tree tensor networks (i.e., hierarchical Tucker formats) with fixed ranks is known to form a manifold \cite{Uschmajew2013a}. This manifold is invariant under the natural action of $\mathrm{GL}(l_1)\times\dots\times\mathrm{GL}(l_D)$. Since multilinear rank is also invariant under this action \cite{Landsberg2012}, all concise (i.e., multilinear rank equals the dimension of the ambient space) tree-based tensor networks are valid Tucker core structures. This includes all aforementioned Tucker, tensor trains, and hierarchical Tucker decompositions in $\R^{l_1 \times\dots\times l_D}$ of multilinear rank $(l_1,\dots,l_D)$. In particular, $\mathbb{R}_\star$ is a valid Tucker core structure.

We will be interested in additive decompositions whose elementary terms lie in structured Tucker manifolds, called \textit{structured block term decompositions (SBTDs)}.
\begin{definition}[Structured block term decomposition]
    An SBTD of the tensor $\mathpzc{A} \in \mathbb{R}^{n_1 \times \dots \times n_D}$ associated with the $\mathcal{M}_r$-structured Tucker manifolds $\mathcal{M}_r^{n_1,\dots,n_D}$ is a decomposition of the form
    $
        \mathpzc{A} = \mathpzc{A}_1 + \dots + \mathpzc{A}_R
    $
    with $\mathpzc{A}_r \in \mathcal{M}_r^{n_1,\dots,n_D}$ for $r=1,\dots,R$.
\end{definition}

Any sum mixing rank-$1$ tensors, Tucker decompositions, tensor trains decompositions, and hierarchical Tucker decompositions is thus an SBTD.

\section{The geometry of the structured Tucker manifold}
\label{sec:geometry}

The condition number of join decompositions from \cite{Breiding2018a} requires that the summands in \cref{eq:decomposition} live on manifolds. Therefore, we first derive the geometric properties of the manifolds involved in the decomposition. We prove \cref{prop_itsamanifold}, which shows that the $\mathcal{M}^{n_1,\dots,n_D}$ in \cref{def:GTucker} are indeed manifolds. We also derive an expression for its tangent space.
The proofs of these statements are standard computations in differential geometry, similar to those of \cite{Uschmajew2013a}.

The following result establishes the differential structure of our manifolds.

\begin{proposition}
\label{prop:quotientMfd}
Let $\mathcal{M}$ be a Tucker core structure as in \cref{def_Tuckercore}. Define the manifolds
\[
\widetilde{\mathcal{M}} := \mathcal{M} \times \mathbb{R}_\star^{n_1 \times l_1} \times \dots \times \mathbb{R}_\star^{n_D \times l_D}
\quad\text{and}\quad
\mathcal{G} := \mathrm{GL}(l_1) \times \dots \times \mathrm{GL}(l_D)
\]
and the group action
\begin{align*}
    \theta: \mathcal{G} \times \widetilde{\mathcal{M}} &\rightarrow \widetilde{\mathcal{M}}\\
    \left( (A_1, \dots, A_D), (\mathpzc{C}, U_1, \dots, U_D) \right) &\mapsto
    \left(
        \left(A_1^{-1}, \dots, A_D^{-1}\right)\cdot\mathpzc{C},\,
        U_1 A_1, \dots, U_D A_D
    \right)
.\end{align*}
Then $\widetilde{\mathcal{M}}/\mathcal{G}$ is a quotient manifold with a unique smooth structure so that the quotient map $[\cdot]_{\mathcal{G}}: \widetilde{\mathcal{M}} \rightarrow \widetilde{\mathcal{M}}/\mathcal{G}$ is a smooth submersion.
\end{proposition}
\begin{proof}
By \cite[Theorem 21.10]{Lee2013}, we only need to verify that the action is smooth, free (i.e., it fixes the identity), and proper. The first two properties are straightforward to check.
To show that it is proper, consider the sequences $\{x_n\}_{n \in \mathbb{N}} \rightarrow x$ in $\widetilde{\mathcal{M}}$ and $\{ \mathfrak{a}_n \}_{n \in \mathbb{N}}$ in $\mathcal{G}$ where $\{ \theta(\mathfrak{a}_n, x_n) \}_{n \in \mathbb{N}}$ converges in $\widetilde{\mathcal{M}}$. By \cite[Proposition 21.5]{Lee2013}, $\theta$ is proper if $\{ \mathfrak{a}_n \}_{n \in \mathbb{N}}$ converges in $\mathcal{G}$. Consider the map $f: \widetilde{\mathcal{M}} \times \widetilde{\mathcal{M}} \rightarrow \mathcal{G}$ taking
\[
(\mathpzc{C},U_1,\dots,U_D), (\widehat{\mathpzc{C}}, \hat{U}_1,\dots, \hat{U}_D)
\mapsto 
(U_1^\dagger \hat{U}_1,\dots,U_D^\dagger \hat{U}_D),
\]
where $U_d^\dagger = (U_d^T U_d)^{-1} U_d^T$ is the Moore-Penrose inverse. Note $f(x_n, \theta(\mathfrak{a}_n, x_n)) = \mathfrak{a}_n$. Furthermore, $f$ is continuous by the continuity of the Moore-Penrose inverse. Since $\{(x_n, \theta(\mathfrak{a}_n, x_n))\}_{n \in \mathbb{N}}$ converges, so does $\{f(x_n, \theta(\mathfrak{a}_n, x_n))\}_{n \in \mathbb{N}} = \{\mathfrak{a}_n\}_{n \in \mathbb{N}}$.
\end{proof}

The tangent space to this quotient manifold is derived next.

\begin{proposition}
\label{prop:tangentQuotientMfd}
Take the manifold $\widetilde{\mathcal{M}} / \mathcal{G}$ as in \cref{prop:quotientMfd} and consider a point $\vb{x} = (\mathpzc{C}, U_1, \dots, U_D) \in \widetilde{\mathcal{M}}$. Complete each $U_d$ to an basis $[U_d \quad U_d^\perp]$ of $\mathbb{R}^{n_d}$. Then
\[
    T_{[\vb{x}]_{\mathcal{G}}}(\widetilde{\mathcal{M}} / \mathcal{G}) \cong \left\{
        \left(
        \dot{\mathpzc{C}}, U_1^\perp B_1, \dots, U_D^\perp B_D
        \right)
    \mid
        \dot{\mathpzc{C}} \in T_{\mathpzc{C}} \mathcal{M},\;
        B_d \in \mathbb{R}^{(n_d - l_d) \times l_d}
    \right\}.
\]
\end{proposition}
\begin{proof}
Define the fibre $\mathcal{F}$ of all $\vb{x}^\prime$ equivalent to $\vb{x}$:
\[
\mathcal{F}_\vb{x} = \left\{\left(
(A_1, \dots, A_D) \cdot \mathpzc{C}, U_1 A_1^{-1},
\dots, U_D A_D^{-1}
\right)
\mid
A_d \in \mathrm{GL}(l_d),\; d = 1,\dots,D
\right\}
.\]
This allows us to define the vertical space as the tangent space to $\mathcal{F}$ at $\vb{x}$:
\begin{equation}
\label{eq:vertSpace}
\mathbb{V}_{\vb{x}} = \left\{
\left(
\sum_{d=1}^D \left( \dot{A}_d \otimes_d \bigotimes_{d' \ne d} \mathds{1}_{l_{d'}} \right) \mathpzc{C},\,
-U_1 \dot{A}_1, \dots, -U_D \dot{A}_D
\right)\mid
\dot{A}_d \in \mathbb{R}^{l_d \times l_d}
\right\}.
\end{equation}
In this expression, we used $T_{A_d} \mathrm{GL}(l_d) \cong \mathbb{R}^{l_d \times l_d}$\cite{Lee2013} for each $d=1,\dots,D$. Now define the horizontal space as
\[
\mathbb{H}_{\vb{x}} := \left\{
\left(
\dot{\mathpzc{C}}, U_1^\perp B_1, \dots, U_d^\perp B_d
\right)
\mid
\dot{\mathpzc{C}} \in T_{\mathpzc{C}} \mathcal{M},\;
B_d \in \mathbb{R}^{(n_d - l_d) \times l_d}
\right\}
.\]
We will show that $\mathbb{V}_{\vb{x}} \oplus \mathbb{H}_{\vb{x}} = T_{\vb{x}} \widetilde{\mathcal{M}}$. First, we verify that the intersection is trivial.
Take $\xi \in \mathbb{V}_{\vb{x}}$, parametrised as in \cref{eq:vertSpace}. If also $\xi \in  \mathbb{H}_{\vb{x}}$, by construction of~$U^d_\perp$, it must hold that all $\dot{A}_d$ in the parametrisation of $\xi$ are zero and hence $\xi = \vb{0}$.

Next, we show that the sum is $T_{\vb{x}} \widetilde{\mathcal{M}}$.
We know $\mathcal{M}$ is invariant under the application of $\mathrm{GL}(l_1) \times \dots \times \mathrm{GL}(l_D)$. Therefore, for any $\dot{A}_d \in \mathbb{R}^{l_d \times l_d}$ for $d = 1,\dots,D$, there exist curves over $\mathcal{M}$ of the form
$
\gamma(t) = \left(A_1(t), \dots, A_D(t)\right) \cdot \mathpzc{C}
$
with $A_d(0) = \mathds{1}_{l_d}$ and $\frac{\mathrm{d}}{\mathrm{d}t}\vert_{t=0} A_d(t) = \dot{A}_d$. Hence, all tensors of the form
\[
\gamma^\prime(0) = \sum_{d=1}^D \left( \dot{A}_d \otimes_d \bigotimes_{d' \ne d} \mathds{1}_{l_{d'}} \right) \mathpzc{C}
\quad \text{with} \quad
\dot{A}_d \in \mathbb{R}^{l_d \times l_d} ,\quad d = 1,\dots,D,
\]
are tangent to $\mathcal{M}$ at $\vb{x}$. Because of this, it is easy to check that
\[
\mathbb{V}_{\vb{x}} \oplus \mathbb{H}_{\vb{x}}
= T_{\mathpzc{C}}\mathcal{M} \times \mathbb{R}^{n_1 \times l_1} \times \dots \times \mathbb{R}^{n_D \times l_D}
= T_{\vb{x}} \widetilde{\mathcal{M}}
.\]
By the general theory of quotient manifolds, this establishes $\mathbb{H}_{\vb{x}} \cong T_{[\vb{x}]_{\mathcal{G}}} \left(\widetilde{\mathcal{M}} / \mathcal{G} \right)$, where the isomorphism is the unique horizontal lift \cite[Section 3.5.8]{Absil2008}.
\end{proof}

We have established that Tucker decompositions with a structured core form a smooth manifold. By definition, a point on an $\mathcal{M}$-structured Tucker manifold corresponds to a Tucker decomposition that is unique up to basis transform. We now have all the tools we need to show that $\mathcal{M}^{n_1,\dots,n_D}$ is a manifold. We do this next.

\begin{proposition}
\label{prop:GTuckerIsMfd}
Let $\widetilde{\mathcal{M}}/\mathcal{G}$ be as in \cref{prop:quotientMfd} and let $\mathcal{M}^{n_1,\dots,n_D}$ be the $\mathcal{M}$-structured Tucker manifold. Then $\mathcal{M}^{n_1,\dots,n_D}$ is a smooth embedded submanifold of $\mathbb{R}^{n_1 \times \dots \times n_D}$ and the following is a diffeomorphism:
\begin{align*}
    \Phi: \widetilde{\mathcal{M}}/\mathcal{G} & \rightarrow \mathcal{M}^{n_1,\dots,n_D} \\
    [(\mathpzc{C}, U_1, \dots, U_D)]_{\mathcal{G}} &\mapsto (U_1, \dots, U_D) \cdot \mathpzc{C}
.\end{align*}
Moreover, the tangent space to $\mathcal{M}^{n_1,\dots,n_D}$ at $(U_1, \dots, U_D) \cdot \mathpzc{C}$ is generated by all tensors
\begin{equation}
    \label{eq:TspaceGTucker}
    (U_1, \dots, U_D) \cdot \dot{\mathpzc{C}}
    + \sum_{d=1}^D \left( \dot{U}_d \otimes_d \bigotimes_{d' \ne d} U_{d'} \right) \mathpzc{C}
\end{equation}
with $\dot{\mathpzc{C}} \in T_{\mathpzc{C}} \mathcal{M}$
and $U_d^\dagger \dot{U}_d = \vb{0}_{l_d \times l_d}$ for all $d=1,\dots,D$.
\end{proposition}
\begin{proof}
    By \cite[Proposition 5.2]{Lee2013}, the first claim holds if $\Phi$ is both a homeomorphism and a smooth immersion.
    First, we show that it is a homeomorphism. Note that $\Phi$ is a bijection because $\mathcal{M}^{n_1,\dots,n_D}$ is precisely the set of all tensors with a Tucker decomposition where the core is in $\mathcal{M}$. Since $\Phi$ is induced by a polynomial map, it is also continuous. To show that $\Phi^{-1}$ is continuous, consider the maps
    \begin{align*}
        V_d: \mathcal{M}^{n_1,\dots,n_D} &\rightarrow \mathrm{Gr}(n_d, l_d) \\
        (U_1, \dots, U_D) \cdot \mathpzc{C} &\mapsto [U_d]_{\mathrm{GL}(l_d)}
    \end{align*}
    where $\mathrm{Gr}(n_d, l_d) \cong \mathbb{R}^{n_d \times l_d}_\star / \mathrm{GL}(l_d)$ is the Grassmannian of $n_d$-dimensional linear spaces in $\mathbb R^{l_d}$ \cite{Absil2008}. That is, $V_d(\mathpzc{X})$ is the column span of its $d$th flattening $\mathpzc{X}_{(d)}$.

We will demonstrate continuity of $V_d$ at any $\mathpzc{X} \in \mathcal{M}^{n_1,\dots,n_D}$ by showing that any open neighbourhood $\mathcal{V}$ of $V_d(\mathpzc{X})$ contains the image of a neighbourhood of $\mathpzc{X}$ \cite[Theorem 18.1]{Munkres2014}. By the definition of the quotient topology, $\mathcal{V} = [\mathcal{U}]_{\mathrm{GL}(l_d)}$ for some open neighbourhood $\mathcal{U} \subseteq \mathbb{R}^{n_d \times l_d}_\star$ of $U_d$, where $U_d$ is any representative of $V_d(\mathpzc{X})$. Furthermore, for some ball $B_\varepsilon(U_d)$ of radius $\varepsilon$ centered at $U_d$, we have $[B_\varepsilon(U_d)]_{\mathrm{GL}(l_d)} \subseteq [\mathcal{U}]_{\mathrm{GL}(l_d)} = \mathpzc{V}$.

Now we exploit the liberty of choosing the representative $U_d$.
    Observe that $V_d(\mathpzc{X})$ is the span of~$l_d$ columns of $\mathpzc{X}_{(d)}$. In other words,
    there exists a column selection operator $P_d \in \mathbb{R}^{(\prod_{d' \ne d} n_{d'}) \times l_d}$, so that $V_d(\mathpzc{X}) = [\mathpzc{X}_{(d)} P_d]_{\mathrm{GL}(l_d)}$. By the upper semicontinuity of matrix rank, 
    there exists $0 < \delta < \varepsilon$ so that any perturbation to $\mathpzc{X}$ of norm less than $\delta$ does not change the rank of $\mathpzc{X}_{(d)}$ or $\mathpzc{X}_{(d)}P_d$. Hence, $V_d(\tilde{\mathpzc{X}}) = [\tilde{\mathpzc{X}}_{(d)} P_d]_{\mathrm{GL}(l_d)}$ for any $\tilde{\mathpzc{X}}$ in a ball $B_{\delta}(\mathpzc{X})$ of radius $\delta$. Because $\norm{\tilde{\mathpzc{X}}_{(d)} P_d - \mathpzc{X}_{(d)} P_d} < \delta < \varepsilon$, we have $V_d(B_{\delta}(\mathpzc{X})) \subseteq [B_{\varepsilon}(\mathpzc{X}_{(d)} P_d)]_{\mathrm{GL}(l_d)}$, which proves continuity of $V_d$.

    For any $\mathpzc{X} \in \mathcal{M}^{n_1,\dots,n_D}$, let $V_d(\mathpzc{X}) = [U_d]_{\mathrm{GL}(l_d)}$ for each $d$. It can be verified that the following is independent of the representatives $U_d$:
    \[
    \Psi(\mathpzc{X}) := [((U_1^\dagger, \dots, U_D^\dagger) \cdot \mathpzc{X}, U_1, \dots, U_D)]_{\mathcal{G}}
    .\]
    The right-hand side is the Tucker decomposition of $\mathpzc{X}$, which is unique up to the action of $\mathpzc{G}$. Hence, $\Psi = \Phi^{-1}$.  This shows that $\Phi^{-1}$ is the composition of continuous maps: $V_d$ for each $d$, the Moore-Penrose inverse, multilinear multiplication, and the canonical projection map. Hence, $\Phi^{-1}$ is continuous.

    Next, we show that $\Phi$ is an immersion, i.e., that its derivative maps a basis to a basis, in which case $T_{\mathpzc{X}}\mathcal{M}^{n_1,\dots,n_D}$ is the image of $\mathrm{d}\Phi$. Fix a basis $\mathscr{B}_{0}$ of $T_{\mathpzc{C}} \mathcal{M}$ and, for each $d = 1,\dots,D$, a basis $\mathscr{B}_d$ of all $\dot{U}_d \in \mathbb{R}^{n_d \times l_d}$ so that $U_d^\dagger \dot{U}_d = 0$.
    By \cref{prop:tangentQuotientMfd}, the tangent space of $\widetilde{\mathcal{M}}/\mathcal{G}$ can be considered as a product space generated by the canonical product basis derived from $\mathscr{B}_0,\dots,\mathscr{B}_D$.

    Applying $\mathrm{d}\Phi$ to this basis of $\widetilde{\mathcal{M}}/\mathcal{G}$ gives tangents of the form
    \[
    T_0 := \left\{
    (U_1, \dots, U_D) \cdot \dot{\mathpzc{C}}
    \right\}
    \quad\text{and}\quad
    T_{d} := \left\{
    \left(U_1, \dots, U_{d-1}, \dot{U}_d, U_{d+1}, \dots, U_D \right) \cdot \mathpzc{C}
    \right\}
    \]
    in which $\mathpzc{C} \in \mathscr{B}_0$ and $\dot{U}_d \in \mathscr{B}_d$, $d = 1,\dots,D$.
    Note that the sets $T_i$ and $T_j$ with $i \ne j$ are pairwise orthogonal due to the constraint on $\dot{U}_d$. Since $(U_1 \otimes \dots \otimes U_D)$ has full rank, $T_0$ is linearly independent. The tangents in the set $T_d$ with $d \ge 1$ are tensors whose $d$th unfolding is
    \begin{equation}
    \label{eq:unfoldingTangents}
        \dot{U}_d \mathpzc{C}_{(d)} (U_1 \otimes \dots \otimes U_{d-1} \otimes U_{d+1} \otimes \dots \otimes U_D)^T.
    \end{equation}
    Recall from \cref{def:GTucker} that $\mathpzc{C}_{(d)}$ and all $U_i$ have full row rank. For a set of linearly independent matrices $\dot{U}_d$, all matrices \cref{eq:unfoldingTangents} are linearly independent. This shows that $\Phi$ is an immersion. By \cite[Proposition 5.2]{Lee2013}, $\mathcal{M}^{n_1,\dots,n_D}$ is a manifold and $\Phi$ is a diffeomorphism. By \cite[Theorem 4.14]{Lee2013}  $T_{\mathpzc{X}}\mathcal{M}^{n_1,\dots,n_D}$ is the image of $\mathrm{d}\Phi$.
\end{proof}

Note that \cref{prop_itsamanifold} is a corollary of the previous statement.

\section{Computing the condition number}
\label{sec:computeCond}
Having shown that the structured Tucker decompositions form a manifold, we can investigate their condition number using the tools from \cite{Breiding2018a}.
For this, we first derive an orthonormal basis of the structured Tucker manifold, so that the condition number can be computed with efficient algorithms from linear algebra using \cref{eq:computeJoinCond} below. We present some examples, as well as useful estimates of the condition number of SBTDs.

\subsection{A direct algorithm}
\label{sec:algorithm}
Let
$\mathcal{M}_1^{n_1,\dots,n_D}, \dots, \mathcal{M}_R^{n_1,\dots,n_D}$ be structured Tucker manifolds, and recall the addition map
\[
\Sigma: \mathcal{M}_1^{n_1,\dots,n_D}\times \dots\times \mathcal{M}_R^{n_1,\dots,n_D}\to \mathbb R^{n_1\times \dots\times n_D},\;(\mathpzc{A}_1,\dots,\mathpzc{A}_R)\mapsto \mathpzc{A}_1+\dots+\mathpzc{A}_R
\]
from the introduction.
Computing an SBTD translates to finding a decomposition $(\mathpzc{A}_1,\dots,\mathpzc{A}_R)$ so that $\Sigma(\mathpzc{A}_1,\dots,\mathpzc{A}_R) = \mathpzc{A}$. The condition number $\kappa^{\mathrm{SBTD}}(\mathpzc{A}_1,\dots, \mathpzc{A}_R)$ from \cref{eq:defcond} is computed as follows \cite{Breiding2018a}. For $r = 1,\dots,R$, compute orthonormal bases of $T_{\mathpzc{A}_r}\mathcal{M}_r^{n_1,\dots,n_D}$, the tangent space to $\mathcal{M}_r^{n_1,\dots,n_D}$ at~$\mathpzc{A}_r$. The basis vectors are the columns of matrices $T_r$.
Then, the so-called \textit{Terracini matrix} is constructed as
\begin{equation}
    \label{eq:Terracini}
T_{\mathpzc{A}_1,\dots, \mathpzc{A}_R} := \begin{bmatrix} T_1 & \dots & T_R \end{bmatrix}.
\end{equation}
 The condition number satisfies
\begin{equation}
    \label{eq:computeJoinCond}
\kappa^{\mathrm{SBTD}}(\mathpzc{A}_1, \dots, \mathpzc{A}_R ) = \frac{1}{\sigma_{\min}(T_{\mathpzc{A}_1,\dots, \mathpzc{A}_R})},
\end{equation}
where $\sigma_{\min}(A) = \sigma_{\min\{m,n\}}(A)$ denotes the smallest singular value of $A \in \R^{m \times n}$. Thus, the computation of~$\kappa^{\mathrm{SBTD}}$ requires orthonormal bases of the tangent spaces to the structured Tucker manifolds. We explain this in the next proposition.

Recall that the compact higher-order singular value decomposition (HOSVD) \cite{Tucker1966,DeLathauwer2000} is an orthogonal Tucker decomposition
$
 \mathpzc{X} = (U_1, \dots, U_D) \cdot \mathpzc{C}
$
with $U_d \in \R^{n_d \times l_d}$ a basis of left singular vectors of $\mathpzc{X}_{(d)}$ corresponding to the nonzero singular values. In particular, $U_d^T U_d = \mathds{1}_{l_d}$ and the columns of $U_d$ span the column span of $\mathpzc{X}_{(d)}$. The core tensor $\mathpzc{C}$ is the orthogonal projection of $\mathpzc{X}$ onto the orthonormal basis $U_1 \otimes \dots \otimes U_D$: $\mathpzc{C} = (U_1^T, \dots, U_D^T) \cdot \mathpzc{X}$. With this terminology in place, we can state the result.

\begin{proposition}
    \label{prop:GTuckerOrthBasis}
    Let $\mathcal{M}^{n_1,\dots,n_D} \subseteq \mathbb{R}^{n_1 \times \dots \times n_D}$ be the $\mathcal{M}$-structured Tucker manifold with Tucker core structure $\mathcal{M} \subseteq \mathbb{R}^{l_1 \times \dots \times l_D}$.
    Assume that we are given a tensor $\mathpzc{X}\in \mathcal{M}^{n_1,\dots,n_D}$ expressed in HOSVD format $\mathpzc{X} = \left( U_1, \dots, U_D \right) \cdot \mathpzc{C}$. Complete each $U_d$ to an orthonormal basis $\begin{bmatrix}U_d & U_d^\perp\end{bmatrix}$ of $\mathbb{R}^{n_d}$. Let $\sigma_{j}^d := \norm{ \vb{e}_j^T \mathpzc{C}_{(d)}}$ and $\hat{\vb{u}}_j^d := (\sigma_j^d)^{-1} \vb{e}_j^{(l_d)}$.
    If $\mathscr{B}_{\mathpzc{C}}$ is an orthonormal basis for $T_{\mathpzc{C}} \mathcal{M}$, the following is an orthonormal basis of the tangent space $T_{\mathpzc{X}} \mathcal{M}^{n_1,\dots,n_D}$:
    \begin{equation}
        \label{eq:GTuckerOrthBasis}
    \mathscr{B}_{\mathpzc{X}} :=
    \left\{
        (U_1, \dots, U_D) \cdot \dot{\mathpzc{C}}
    \right\}
    \cup
    \left\{
        (U_1, \dots, U_{d-1}, U_d^\perp \vb{e}_i (\hat{\vb{u}}_{j}^d)^T, U_{d+1}, \dots, U_D) \cdot \mathpzc{C}
    \right\}
    \end{equation}
    in which $\dot{\mathpzc{C}} \in \mathscr{B}_{\mathpzc{C}}$, $d=1,\dots,D$, $i = 1,\dots,n_d - l_d$ and $j = 1,\dots,l_d$.
\end{proposition}
\begin{proof}
\Cref{eq:TspaceGTucker} for $T_{\mathpzc{X}} \mathcal{M}^{n_1,\dots,n_D}$ suggests a decomposition of the tangent space of the form
\[
T_{\mathpzc{X}} \mathcal{M}^{n_1,\dots,n_D} = \mathbb{T}_0 \oplus \mathbb{T}_1 \oplus \dots \oplus \mathbb{T}_D,
\]
where $\mathbb{T}_0$ contains all tangents of the form $(U_1, \dots, U_D) \cdot \dot{\mathpzc{C}}$ and $\mathbb{T}_d$ with $d = 1,\dots,D$ contains the tangents of the form $(U_1, \dots, U_{d-1}, \dot{U}_d, U_{d+1}, \dots, U_D) \cdot \mathpzc{C}$. As argued in the proof of \cref{prop:GTuckerIsMfd}, this is a decomposition of $T_{\mathpzc{X}}\mathcal{M}^{n_1,\dots,n_D}$ into pairwise orthogonal spaces.

First we verify that \cref{eq:GTuckerOrthBasis} spans $T_\mathpzc{X} \mathcal{M}^{n_1,\dots,n_D}$.
Since we have an orthonormal basis of $T_{\mathpzc{C}} \mathcal{M}$ available, we have
\[
\mathbb{T}_0 = \mathrm{span}
\left\{
(U_1, \dots, U_D) \cdot \dot{\mathpzc{C}}
\mid
\dot{\mathpzc{C}} \in \mathscr{B}_{\mathpzc{C}}
\right\} = (U_1 \otimes \dots \otimes U_D)(T_\mathpzc{C} \mathcal{M})
.\]
For the other $D$ subspaces $\mathbb{T}_d$, we require all $\dot{U}_d$ such that $U_d^T \dot{U}_d = 0$, or equivalently $\dot{U}_d = U^\perp_d B$ for some $B \in \mathbb{R}^{(n_d - l_d) \times l_d}$.
The $\vb{e}_i^{(n_d - l_d)} (\hat{\vb{u}}_j^d)^T$ with $i = 1,\dots,n_d - l_d$ and $j = 1,\dots,l_d$ are a basis of $\mathbb{R}^{(n_d - l_d) \times l_d}$, because they are just a rescaling of the canonical basis $\vb{e}_i^{(n_d - l_d)} (\vb{e}_j^{(l_d)})^T$.
Substituting each of these for $B$, we get a basis of all allowed $\dot{U}_d$. This parametrises all of $\mathbb{T}_d$ as
\[
\mathbb{T}_d = \mathrm{span}
\left\{
(U_1, \dots, U_{d-1}, U_d^\perp \vb{e}_i (\hat{\vb{u}}_{j}^d)^T, U_{d+1}, \dots, U_D) \cdot  \mathpzc{C}
\right\}_{i=j=1}^{i=n_d-l_d, j=l_d}
.\]
Hence, the proposed basis $\mathscr{B}_{\mathpzc{X}}$ generates $\mathbb{T}_0 \oplus \mathbb{T}_1 \oplus \dots \oplus \mathbb{T}_D$.

We have yet to verify that the proposed basis is orthonormal.
We already know that $\mathbb{T}_0,\dots,\mathbb{T}_D$ are pairwise orthogonal. It thus suffices to show that the bases we constructed for each of these spaces separately are orthonormal.
The basis for $\mathbb{T}_0$ is orthonormal because $\mathscr{B}_{\mathpzc{C}}$ is orthonormal and $U_1 \otimes \dots \otimes U_D$ is an orthonormal tensor product basis.

For the basis of $\mathbb{T}_d$ with $d \ge 1$, we use the fact that $\mathpzc{X}$ is in HOSVD format. This ensures that an HOSVD of $\mathpzc{C}$ is $(\mathds{1}, \dots, \mathds{1}) \cdot \mathpzc{C}$.
In other words, $\mathpzc{C}_{(d)}$ has singular values $\sigma_j^d$ as defined above and its corresponding left singular vectors are $\vb{e}_j$ \cite{DeLathauwer2000}. Hence, the transpose of its $j$th  right singular vector is $(\vb{v}_j^d)^T := (\sigma_j^d)^{-1} \vb{e}_j^T \mathpzc{C}_{(d)} = (\hat{\vb{u}}_j^d)^T \mathpzc{C}_{(d)}$.
With this in mind, we calculate the inner products between the basis vectors of $\mathbb{T}_d$:
\begin{align}
    \label{eq:orthBasisInnerProducts}
\bigl\langle
&(U_1, \dots, U_{d-1}, U_d^\perp \vb{e}_i (\hat{\vb{u}}_j^d)^T, U_{d+1}, \dots, U_D)\cdot \mathpzc{C},\\\nonumber
&\qquad (U_1, \dots, U_{d-1}, U_d^\perp \vb{e}_{i'} (\hat{\vb{u}}_{j'}^d)^T, U_{d+1}, \dots, U_D)\cdot \mathpzc{C}
\bigr\rangle\\\nonumber
= \;&\mathrm{Trace}( U_d^\perp \vb{e}_{i'}(\hat{\vb{u}}_{j'}^d)^T \mathpzc{C}_{(d)} \mathpzc{C}_{(d)}^T \hat{\vb{u}}_{j'}^d \vb{e}_{i'}^T(U_d^\perp)^T )\\\nonumber
=\;&
 \langle
    U_d^\perp \vb{e}_{i},
    U_d^\perp \vb{e}_{i'}
\rangle\, \langle
    \vb{v}_j^d,
    \vb{v}_{j'}^d
\rangle
.\end{align}
If $i=i'$, the right-hand side is the inner product between two right singular vectors of~$\mathpzc{C}_{(d)}$, which is $\delta_{jj'}$, the Kronecker delta. Otherwise, it is zero due to the orthogonality of the columns of $U_d^\perp$.
This ensures that our basis of $\mathbb{T}_d$ is orthogonal, which completes the proof.
\end{proof}

Now we can compute the condition number of several decompositions using the formula in \cref{eq:computeJoinCond}. Consider the following examples.

\paragraph{Example 1 (BTD)} For a BTD with block terms of multilinear rank $(l_1^r, \dots, l_D^r)$, where $r=1,\dots,R$, we can apply \cref{def:GTucker} in which the Tucker core structure $\mathcal{M}$ is the submanifold of tensors in $\mathbb{R}^{l_1^r \times \dots \times l_D^r}$ with multilinear rank $(l_1^r, \dots, l_D^r)$. Since~$\mathcal{M}$ is an open subset of $\mathbb{R}^{l_1^r \times \dots \times l_D^r}$, the canonical basis of $\mathbb{R}^{l_1^r \times \dots \times l_D^r}$ is an orthonormal basis of the tangent space to $\mathcal{M}$ at any point.
The algorithm to compute the condition number $\kappa^{\mathrm{BTD}}$ is as follows. For each term $\mathpzc{A}_r$ in the BTD, compute its compact HOSVD $(U_1^r, \dots, U_D^r) \cdot \mathpzc{C}_r$. An orthonormal basis of the tangent space to the Tucker manifold is given by the columns of
{\small\[
T_{\mathpzc{A}_r} :=
\left[
\bigotimes_{d=1}^D U^r_d
\quad
\left[
    (U^r_1 \otimes\dots\otimes U_{d-1}^r \otimes U_d^{r\perp} \vb{e}_i (\hat{\vb{u}}_d^{rj})^T \otimes U_{d+1}^r \otimes\dots\otimes U^r_D) \mathpzc{C}_r
\right]_{d,i,j=1}^{D, m_d - l^r_d, l^r_d}
\right],
\]}%
where $\hat{\vb{u}}_d^{rj}$ and $U_d^{r\perp}$ are as in \cref{prop:GTuckerOrthBasis}.
The condition number of the BTD with terms $\mathpzc{A}_1,\dots,\mathpzc{A}_R$ can then be computed by applying \cref{eq:computeJoinCond}:
\[
\kappa^{\mathrm{BTD}}(\mathpzc{A}_1,\dots,\mathpzc{A}_R)
= \sigma_{\min}(T_{\mathpzc{A}_1,\dots,\mathpzc{A}_R})^{-1}, \quad
T_{\mathpzc{A}_1,\dots,\mathpzc{A}_R} = \begin{bmatrix} T_{\mathpzc{A}_1} & \dots & T_{\mathpzc{A}_R} \end{bmatrix}
.\]

\paragraph{Example 2 (CPD)} This case was studied in \cite{Breiding2018a}. By applying the Tucker core structure $\mathcal{M} := \mathbb{R} \setminus \{0\}$ to \cref{def:GTucker}, we get the Segre manifold of rank-1 tensors. If $\mathpzc{A}_r = \lambda \vb{u}_1^r \otimes \dots \otimes \vb{u}_D^r$ is a rank-1 tensor with $\norm{\vb{u}_1} = \dots = \norm{\vb{u}_D} = 1$, \cref{prop:GTuckerOrthBasis} gives the following familiar basis:
\[
T_{\mathpzc{A}_r} := \left[
    \vb{u}_1^r \otimes \dots \otimes \vb{u}_D^r
    \quad \left[
        \vb{u}_1^r \otimes\dots\otimes \vb{u}_{d-1}^r \otimes U_d^{r\perp} \otimes \vb{u}_{d+1}^r \otimes\dots\otimes \vb{u}_D^r
    \right]_{d=1}^D
\right],
\]
where $U_d^{r\perp}$ is an orthonormal basis for the complement of $\vb{u}_d^r$ for each $d$. The condition number $\kappa^{\mathrm{CPD}}$ can be computed in a similar fashion as for the BTD.

\subsection{Examples of well and ill-conditioned SBTDs}
In this subsection, we present some qualitative properties that determine the condition number of the SBTD.
As a general rule, the tensor subspace in which the summands live already gives some information about the condition number. For instance, one instance where the condition number is perfect is when the subspaces $U_d^r$ in the Tucker decompositions are pairwise orthogonal. This can be considered as the SBTD equivalent of an orthogonally decomposable (odeco) tensor \cite{Kolda2001}. In such cases, tangent spaces are pairwise orthogonal, so that the following result holds.

\begin{proposition}
Suppose $\mathpzc{A}_1 + \dots + \mathpzc{A}_R$ is an SBTD with $\mathpzc{A}_r = (U_1^r, \dots, U_D^r) \cdot \mathpzc{C}_r$ in HOSVD form for $r=1,\dots,R$. Assume that $(U^{r_1}_d)^T U^{r_2}_d = 0$ for each $d$ and each $r_1 \ne r_2$.
Then $\kappa^{\mathrm{SBTD}}(\mathpzc{A}_1,\dots,\mathpzc{A}_r) = 1$.
\end{proposition}
\begin{proof}
The columns of the Terracini matrix can be grouped into orthonormal bases of $\Span\{\bigotimes_{d=1}^D U_d^r\}$ and $\Span \{ U_d^{r\perp} \otimes_d \bigotimes_{d' \ne d} U_{d'}^r \}$ for all $r=1,\dots,R$ and for all $d=1,\dots,D$. By assumption, these spaces are all pairwise orthogonal. Since their bases are orthonormal, all columns of the Terracini matrix are orthonormal.
\end{proof}

The fact that this result does not depend on the cores $\mathpzc{C}_r$ may be surprising if the problem is not considered geometrically. $\mathpzc{C}_r$ may be arbitrarily close to having a multilinear rank lower than the specified $(l^r_1, \dots, l^r_D)$ without it affecting the condition number. Despite this, summands which are close to being low multilinear rank are a notorious obstacle in practical algorithms to compute the BTD, the other being correlations between the terms \cite{Navasca2008}. Note that the latter is essentially what the condition number measures. For tensors of lower multilinear rank than $(l^r_1, \dots, l^r_D)$, there are more ways to parametrise it than is accounted for by the usual symmetries. For instance, for a BTD with multilinear ranks $(l_r, l_r, 1)$, the Jacobian of the residual $\mathpzc{A} - \sum_{r=1}^R \mathpzc{A}_r$ with respect to the parameters becomes singular at such points \cite{Sorber2013a}. The ALS algorithm for a general block term decomposition requires solving a system which also becomes singular at the boundary \cite{DeLathauwer2008a}.

However, if the summands are considered as one geometric object, summands close to tensors of lower multilinear rank are not an issue, which explains why it is still reasonable to expect the condition number to be 1 even near the boundary. This suggests that Riemannian optimisation algorithms to compute the BTD could have a significant advantage in these cases, as the convergence rate tends to be related to the condition number. This is analogous to the case of the CPD, where \cite{Breiding2018a} showed experimentally and theoretically that classic flat optimization methods perform worse if the CPD contains summands of small norm---the analogous situation to a lower multilinear rank in $\mathcal{M}$-structured Tucker decompositions---while Riemannian optimisation methods that treat the summands as one geometric object did not suffer as much.

In general, the condition number of any SBTD can be upper bounded by the condition number of the corresponding BTD, and for the latter we can get a useful lower bound for the condition number. We show this in the next proposition.
\begin{proposition}
    \label{prop:condEstimate}
Given any SBTD $\mathpzc{A} = \mathpzc{A}_1 + \cdots +\mathpzc{A}_R$, we can also regard it as a BTD of $\mathpzc{A}$.
The condition numbers satisfy
\[
\kappa^{\mathrm{BTD}}(\mathpzc{A}_1, \dots, \mathpzc{A}_R)
\ge
\kappa^{\mathrm{SBTD}}(\mathpzc{A}_1, \dots, \mathpzc{A}_R).
\]
If the terms $\mathpzc{A}_r = (U_1^r, \dots, U_D^r) \cdot \mathpzc{C}_r$ are in HOSVD form for $r=1,\dots,R$, then
\begin{equation*}
\kappa^{\mathrm{BTD}}(\mathpzc{A}_1, \dots, \mathpzc{A}_R)
\ge
\sigma_{\min}\left(
    \left[U_1^r \otimes\dots\otimes U_D^r
    \right]_{r=1}^R
\right)^{-1}
.\end{equation*}
\end{proposition}
\begin{proof}
Assume $\kappa^{\mathrm{BTD}}(\mathpzc{A}_1,\dots,\mathpzc{A}_R) < \infty$, since otherwise the statement is trivially true. For each $r=1,\dots,R$, the $r$th structured Tucker manifold $\mathcal{M}_r^{n_1,\dots,n_D}$ of the SBTD is a subset of the manifold $\mathcal{N}_r^{n_1,\dots,n_D}$ of tensors of fixed multilinear rank. By assumption, the addition map $\Sigma(\mathpzc{A}_1,\dots,\mathpzc{A}_R) = \mathpzc{A}_1 +\dots + \mathpzc{A}_R$ of the BTD has a local inverse function $\Sigma^{-1}_{\mathpzc{A}_1,\dots,\mathpzc{A}_R}$, defined on a neighbourhood $\mathcal{I} \subseteq \Sigma(\mathcal{N}_1^{n_1,\dots,n_D},\dots,\mathcal{N}_R^{n_1,\dots,n_D})$ of $\mathpzc{A}$. On the other hand,
 for any $\widetilde{\mathpzc{A}} \in \mathcal{I}' := \mathcal{I} \cap \Sigma(\mathcal{M}_1^{n_1,\dots,n_D},\dots,\mathcal{M}_R^{n_1,\dots,n_D})$, the (locally unique) SBTD of $\widetilde{\mathpzc{A}}$ is $\Sigma_{\mathpzc{A}_1,\dots,\mathpzc{A}_R}^{-1}(\widetilde{\mathpzc{A}})$. 
The condition numbers of the BTD and SBTD are \cref{eq:defcond} applied to $\Sigma^{-1}_{\mathpzc{A}_1,\dots,\mathpzc{A}_R}$ and the restriction of $\Sigma^{-1}_{\mathpzc{A}_1,\dots,\mathpzc{A}_R}$ onto $\mathcal{I}'$, respectively. Since $\mathcal{I}' \subseteq \mathcal{I}$, the first statement follows.

For the second assertion, observe that the columns of $\left[U_1^r \otimes\dots\otimes U_D^r
\right]_{r=1}^R$ are a subset of the columns of the Terracini matrix $T$ of the BTD. By \cite[Theorem 8.1.7]{Golub2013}, $\sigma_{\min}(\left[U_1^r \otimes\dots\otimes U_D^r
\right]_{r=1}^R) \ge \sigma_{\min}(T)$. The result follows from \cref{eq:computeJoinCond}.
\end{proof}

The second item in the above proposition shows that,
\[
\text{if } \ker \left[U_1^r \otimes\dots\otimes U_D^r
\right]_{r=1}^R \neq \{0\}, \text{ then } \kappa^{\mathrm{BTD}}(\mathpzc{A}_1, \dots, \mathpzc{A}_R) = \infty.
\]
It is easy to see on an intuitive level why the condition number must be infinite at these points. If the spaces intersect, there exist cores $\widetilde{\mathpzc{C}}_r$ with $r=1,\dots,R$ so that
$
    \sum_{r=1}^R (U_1^r \otimes\dots\otimes U_D^r) \widetilde{\mathpzc{C}}_r = 0
$
and not all $\widetilde{\mathpzc{C}}_r = 0$. Then we can define the following curve:
$
    \gamma(t) := (\gamma_1(t),\dots,\gamma_R(t)),
$
where
$\gamma_r(t) := ((U_1^r \otimes\dots\otimes U_D^r) ( \mathpzc{C}_r + t \widetilde{\mathpzc{C}}_r)),$
so that $\gamma_1(t) + \dots + \gamma_R(t) = \mathpzc{A}$ is constant. Therefore, there exists a smooth curve segment in the neighborhood of $t=0$ of equivalent BTDs of $\mathpzc{A}$, which means finding the BTD of $\mathpzc{A}$ is an ill-posed problem. In other words, a zero-norm perturbation of $\mathpzc{A}$ is sufficient to get actually different decompositions; hence, the condition number \cref{eq:defcond} is $\infty$.

This leads to the following observation: If the condition number is finite, the BTD can be determined purely from subspace information. That is, suppose that for a given tensor $\mathpzc{A}$, only the subspaces $U_1^r \otimes\dots\otimes U_D^r$ are computed for each $r$th block term, with $r=1,\dots,R$. Because the subspaces do not intersect, the cores $\mathpzc{C}_r$ can be uniquely recovered from the linear system
$
\mathpzc{A} = \sum_{r=1}^R (U_1^r \otimes\dots\otimes U_D^r) \mathpzc{C}_r
$
This is exactly the principle behind the variable projection methods in \cite{Olikier2018}.

It is worth pointing out that the second lower bound from \cref{prop:condEstimate} is not necessarily sharp. To see this, let $\mathcal{M}^{n_1,\dots,n_D}$ be a structured Tucker manifold and consider the SBTD $\mathpzc{A} = \mathpzc{A}_1 + \mathpzc{A}_2$ with the two summands having Tucker compressions $\mathpzc{A}_1 = (U_1,\dots,U_D) \cdot \mathpzc{C} \in \mathcal{M}^{n_1,\dots,n_D}$ and $\mathpzc{A}_2 = (V_1,U_2,\dots,U_D) \cdot \mathpzc{C} \in \mathcal{M}^{n_1,\dots,n_D}$. We assume that $U_1^T V_1 = 0$ and
we define the two curves $\gamma_1(t) := (U_1 + t V_1, U_2,\dots,U_D) \cdot \mathpzc{C}$ and $\gamma_2(t) := ((1-t) V_1, U_2, \dots, U_D) \cdot \mathpzc{C}$. Assuming $t$ is small enough, we have that $\gamma_1(t), \gamma_2(t) \in \mathcal{M}^{n_1,\dots,n_D}$. As in the previous example $\gamma_1(t) + \gamma_2(t) = \mathpzc{A}$, for all $t$. Hence, the condition number is also infinite in this case.
Despite this, the estimated lower bound in \cref{prop:condEstimate} is
$\sigma_{\min}(
    \left[U_1 \otimes U_2 \otimes\dots\otimes U_D, V_1\otimes U_2 \otimes\dots\otimes U_D
    \right]) = 1$.

\section{Invariance of the condition number under Tucker compression}
\label{sec:basicDefs}

Next, we discuss the main contribution of this work.
Our main result was informally stated as \cref{thm:informalSBTDcondInvariance} in the introduction. Here, we present its formal version \cref{thm:SBTDcondInvariance}. These two theorems show that the condition number of computing SBTDs is invariant under Tucker compression. As we explain in the \cref{sec_fastcondition}, this can yield a computationally attractive approach for computing the condition number.

First, we introduce \textit{subspace-constrained SBTDs} as the formal model of decompositions resulting from Bro and Andersson's \cite{Bro1998} compress-decompose-expand approach. Subspace-constrained CPDs were also considered in the recent paper \cite{PTSSEC2021}. Suppose that we want to compute an SBTD of the tensor $\mathpzc{A} \in \R^{n_1 \times \dots \times n_D}$ associated with the $\mathcal{M}_r$-structured Tucker manifolds $\mathcal{M}_r^{n_1,\dots,n_D}$. Then compress-decompose-expand proceeds as follows.

\smallskip
\paragraph{Compress}
\begin{quote}
$\mathpzc{A}$ lives in a minimal tensor product subspace of $\R^{n_1 \times \dots \times n_D}$ (possibly trivial). Its minimal Tucker decomposition  is
$\mathpzc{A} = (Q_1, \dots, Q_D)\cdot \mathpzc{G}$ with core tensor $\mathpzc{G} \in \mathbb{R}^{m_1 \times \dots \times m_D}$ and matrices $Q_d \in \mathbb{R}^{n_d \times m_d}_\star$ for $d = 1,\dots,D$.
The decomposition is minimal if $\mathpzc{G}$ has multilinear rank equal to $(m_1,\dots,m_D)$. It can be computed with a (sequentially) truncated higher-order singular value decomposition \cite{Vannieuwenhoven2012,DeLathauwer2000}.
\end{quote}

\paragraph{Decompose}
\begin{quote}
Compute an SBTD of $\mathpzc{G}$ be
\begin{equation}
    \label{eq:SBTDcore}
\mathpzc{G} = \mathpzc{G}_1 + \dots + \mathpzc{G}_R \quad\text{with}\quad \mathpzc{G}_r \in \mathcal{M}_r^{m_1,\dots,m_D}.
\end{equation}
That is, the $\mathpzc{G}_r$ are $\mathcal{M}_r$-structured Tucker tensors.
\end{quote}

\paragraph{Expand}
\begin{quote}
We expand $\mathpzc{A}_r = (Q_1 U_1^r, \dots,  Q_D U_D^r) \cdot \mathpzc{C}_r$ and find a decomposition $\mathpzc{A} = \mathpzc{A}_1 + \dots + \mathpzc{A}_R$ of $\mathpzc{A}$. In this decomposition the summands are also $\mathcal{M}_r$-structured Tucker tensors: if we have the $\mathcal{M}_r$-structured Tucker decomposition $\mathpzc{G}_r = (U_1^r, \dots, U_D^r) \cdot \mathpzc{C}_r$ such that $U_d^r\in \mathbb R_\star^{m_d\times l_d}$ and $\mathpzc{C}_r$ is a point of
$\mathcal M_r\subset \mathbb{R}^{l_1 \times \dots \times l_D}$ satisfying the assumptions of \cref{def:GTucker}, then for all $r$ and $d$ the matrices $Q_dU_d^r$ are of full rank. Hence $\mathpzc{A}_r = (Q_1 U_1^r, \dots,  Q_D U_D^r) \cdot \mathpzc{C}_r$ is a point of the $\mathcal{M}_r$-structured Tucker manifold $\mathcal{M}_r^{n_1,\dots,n_D}$.
Summarising, the SBTD \cref{eq:SBTDcore} of the compressed tensor $\mathpzc{G}$ can be expanded to an SBTD of $\mathpzc{A}$:
\begin{equation}
    \label{eq:compressedSBTD}
\mathpzc{A} = \mathpzc{A}_1 + \dots + \mathpzc{A}_R
\quad\text{with}\quad
\mathpzc{A}_r \in \mathcal{M}_r^{n_1,\dots,n_D}, \quad
r = 1,\dots,R.
\end{equation}
We call the resulting SBTD of $\mathpzc{A}$ a \textit{subspace-constrained SBTD}, because it is an SBTD all of whose summands are contained in the same tensor subspace $Q_1\otimes\dots\otimes Q_d$ that $\mathpzc{A}$ lives in.
\end{quote}
\smallskip

Given that a subspace-constrained SBTD can be computed by the foregoing compress-decompose-expand approach, it is natural to wonder about the relationship between the condition numbers of $\mathpzc{A}$ and $\mathpzc{G}$. Since $\mathpzc{G}$ lives in a much more constrained space, it seems natural to assume that its condition number could be much lower, similar to the ideas in \cite{Arslan2019}.
In \cref{sec:invprops} below, we prove the following main result about the condition numbers of computing the SBTD \cref{eq:compressedSBTD} of the original tensor and computing the SBTD of the compressed Tucker core \cref{eq:SBTDcore}. A priori, the condition number of the decomposition problem \cref{eq:SBTDcore} is bounded above by the condition number of problem \cref{eq:compressedSBTD}. The next results shows that they are, in fact, always equal.

\begin{theorem}
\label{thm:SBTDcondInvariance}
Let $\mathcal{M}_r\subset\R^{l_1^r\times\dots\times l_D^r}$ be Tucker core structures.
Assume that the tensor $\mathpzc{A} \in \mathbb{R}^{n_1 \times \dots \times n_D}$ has an orthogonal Tucker decomposition $\mathpzc{A} = (Q_1,\dots,Q_D)\cdot\mathpzc{G}$ with $\mathpzc{G}\in \R^{m_1 \times \dots \times m_D}$ and all $Q_d \in \R^{n_d \times m_d}$ having orthonormal columns. Let the subspace-constrained SBTD be $\mathpzc{A} = \mathpzc{A}_1 + \dots + \mathpzc{A}_R$ with $\mathpzc{A}_r \in \mathcal{M}_r^{n_1,\dots,n_D}$ and the SBTD of the Tucker core be $\mathpzc{G} = \mathpzc{G}_1 + \dots + \mathpzc{G}_R$
with $\mathpzc{G}_r \in \mathcal{M}^{m_1,\dots,m_D}_r$, and assume that that they are related by $\mathpzc{A}_r = \left(Q_1, \dots, Q_D \right) \cdot \mathpzc{G}_r$ for each $r=1,\dots, R$.
Then,
\[
\kappa^{\mathrm{SBTD}}(\mathpzc{A}_1, \dots, \mathpzc{A}_R)
= \kappa^{\mathrm{SBTD}}(\mathpzc{G}_1, \dots, \mathpzc{G}_R).
\]
\end{theorem}

Before presenting the proof in \cref{sec:invprops}, let us investigate some consequences of \cref{thm:SBTDcondInvariance}. In the subspace-constrained SBTD there are two levels of multilinear multiplication. The first level is in \cref{def:GTucker} of the structured Tucker decompositions. This level is always written with matrices $U_1^r,\dots,U_D^r$ that depend on the index of the summand. The second level is the multilinear multiplication defining the subspace constraint on the tensor $\mathpzc{A}$. This level is denoted with matrices $Q_1,\dots,Q_D$ and it is the same for all summands. This is summarised in the following diagram:
\[
\begin{tikzcd}
\mathpzc{A}_r\in \mathcal{M}_r^{n_1,\dots,n_D}
         & &&\mathpzc{G}_r\in \mathcal{M}_r^{m_1,\dots,m_D}
         \arrow[lll, "(Q_1 \otimes \dots \otimes Q_D)"]
         & &&\mathpzc{C}_r\in {\mathcal{M}}_r
         \arrow[lll, "(U_1^r \otimes \dots \otimes U_D^r)"].
\end{tikzcd}
\]
It is imperative to note, however, that $(Q_1\otimes \dots \otimes Q_D)\mathcal{M}_r^{m_1,\dots,m_D} \subsetneq \mathcal{M}_r^{n_1,\dots,n_D}$.

When evaluating the sensitivity of a subspace-constrained SBTD \cref{eq:compressedSBTD} via the condition number \cref{eq:defcond}, there are at least four natural sets of perturbations $\mathcal{I}$ to consider. Let $\widetilde{\mathpzc{A}}$ denote the perturbed tensor. It could have resulted from one of the following increasingly restrictive perturbations of the subspace-constrained SBTD $\mathpzc{A}$:
\begin{enumerate}
\item $\mathpzc{A}$ was perturbed with no constraints and $\widetilde{\mathpzc{A}}$ was approximated by the closest SBTD $\widetilde{\mathpzc{A}}\approx \widetilde{\mathpzc{A}}_1+\dots+\widetilde{\mathpzc{A}}_R$ with $\widetilde{\mathpzc{A}}_r \in \mathcal{M}_r^{n_1,\dots,n_D}$;
\item $\mathpzc{A}$ was perturbed so $\widetilde{\mathpzc{A}}$ has an SBTD $\widetilde{\mathpzc{A}} = \widetilde{\mathpzc{A}}_1 + \dots + \widetilde{\mathpzc{A}}_R$ with $\widetilde{\mathpzc{A}}_r \in \mathcal{M}_r^{n_1,\dots,n_D}$;
\item $\mathpzc{A}$ was perturbed so $\mathpzc{A}'$ has a subspace-constrained SBTD $\widetilde{\mathpzc{A}} = (\widetilde{Q}_1, \dots, \widetilde{Q}_D)\cdot\widetilde{\mathpzc{G}}$ with core $\widetilde{\mathpzc{G}}=\widetilde{\mathpzc{G}}_1+\dots+\widetilde{\mathpzc{G}}_R$ and $\widetilde{\mathpzc{G}}_r\in\mathcal{M}_r^{m_1,\dots,m_D}$; or
\item $\mathpzc{A}$ was perturbed inside the fixed subspace $Q_1 \otimes \dots \otimes Q_D$ so $\widetilde{\mathpzc{A}}$ has a subspace-constrained SBTD $\widetilde{\mathpzc{A}} = (Q_1, \dots, Q_D)\cdot\widetilde{\mathpzc{G}}$ with core $\widetilde{\mathpzc{G}}=\widetilde{\mathpzc{G}}_1+\dots+\widetilde{\mathpzc{G}}_R$ and terms in the decomposition $\widetilde{\mathpzc{G}}_r\in\mathcal{M}_r^{m_1,\dots,m_D}$.
\end{enumerate}
Since there are $4$ domains of perturbations we can consider in \cref{eq:defcond}, there are also $4$ associated, a priori distinct, condition numbers. Let us denote the condition number corresponding to the $i$th type of perturbation by $\kappa_i$. Then we have
\begin{align}\label{eq:chainCondInequalities}
 \kappa_1 \ge \kappa_2 = \kappa^{\mathrm{SBTD}}(\mathpzc{A}_1,\dots,\mathpzc{A}_R) \ge \kappa_3 \ge \kappa_4 = \kappa^{\mathrm{SBTD}}(\mathpzc{G}_1,\dots,\mathpzc{G}_R).
\end{align}
However, we already proved in \cite[Corollary 5.5]{Breiding2019} that $\kappa_1 = \kappa_2$, i.e., arbitrary perturbations in combination with a least-squares approximation are no worse than structured perturbations. Combining this with \cref{thm:SBTDcondInvariance} immediately implies the following more formal restatement of \cref{thm:informalSBTDcondInvariance}.
\begin{corollary}
    \label{cor:SBTDcondInvariance}
    Suppose that we have SBTDs $\mathpzc{G} = \mathpzc{G}_1 + \dots + \mathpzc{G}_R \in \mathbb{R}^{m_1 \times \dots \times m_D}$ and $\mathpzc{A} = \mathpzc{A}_1 + \dots + \mathpzc{A}_R \in \mathbb{R}^{n_1 \times \dots \times n_D}$ related by $\mathpzc{A}_r = \left(Q_1 \otimes \dots \otimes Q_D \right) \mathpzc{G}_r$ for each $r=1,\dots, R$.
    If all $Q_d$ have orthonormal columns, then \cref{eq:chainCondInequalities} is an equality.
\end{corollary}

\subsection{Proof of the main result}
\label{sec:invprops}

\Cref{prop:GTuckerOrthBasis} allows us to prove our main result, \cref{thm:SBTDcondInvariance}. Before we do this, we need the following lemma.
\begin{lemma}
    \label{lemma:KroneckerSvdIdentity}
    For any set of matrices $A_k \in \mathbb{R}^{m \times n_k}$ and any set of orthogonal matrices $Q_k \in \mathbb{R}^{p \times p}$ where $k = 1, \dots, K$, the matrices
    \[
    X := \begin{bmatrix}A_1 &\dots& A_K \end{bmatrix}
    \quad\text{and}\quad
    Y := \begin{bmatrix}A_1 \otimes Q_1 &\dots& A_K \otimes Q_K \end{bmatrix}
    \]
    have the same singular values up to multiplicities.
\end{lemma}
\begin{proof}
Define the block diagonal matrix $
D := \text{diag}(\mathds{1}_{n_1} \otimes Q_1, \dots, \mathds{1}_{n_K} \otimes Q_K)
$. Then
$
Y
=
\left[A_1 \otimes \mathds{1}_p \,\,\dots\,\, A_K \otimes \mathds{1}_p \right] D.
$
Since $D$ is orthogonal, $Y$ has the same singular values as $[A_1 \otimes \mathds{1}_p \dots A_K \otimes \mathds{1}_p]$. Up to a permutation of rows and columns, this is $X \otimes \mathds{1}_p$. The proof is completed by applying the singular value property of Kronecker products \cite[section 12.3.1]{Golub2013}.
\end{proof}
\begin{remark}
If each $A_k$ in the above lemma is itself a Kronecker product of at least $d$ matrices and we replace $Y$ by $\begin{bmatrix} Q_1 \otimes_d A_1 & \dots & Q_K \otimes_d A_K \end{bmatrix}$, the statement still holds, because changing the order of the factors only changes their Kronecker product by a permutation of rows and columns \cite[Equation 12.3.1]{Golub2013}.
\end{remark}

Now we can prove that the condition number of the SBTD is invariant under Tucker compression.

\begin{proof}[Proof of \cref{thm:SBTDcondInvariance}]
For the SBTDs $\mathpzc{A} = \mathpzc{A}_1+\dots+\mathpzc{A}_R$ and $\mathpzc{G} = \mathpzc{G}_1+\dots+\mathpzc{G}_R$ we denote their associated Terracini matrices by $T_{\mathpzc{A}_1,\dots,\mathpzc{A}_R}$ and $T_{\mathpzc{G}_1,\dots,\mathpzc{G}_R}$, respectively (see \cref{eq:Terracini}).
Our strategy is assembling the Terracini matrices in an appropriate way and so that we can compare their singular values.

For each $r=1,\dots,R$, we apply \cref{prop:GTuckerOrthBasis} to obtain a basis  $\mathscr{B}_{\mathpzc{G}_r}$ for $T_{\mathpzc{G}_r}\mathcal{M}_{r}^{m_1,\dots,m_D}$. That is, we write $\mathpzc{G}_r$ in HOSVD form $(U_1^r,\dots,U_D^r) \cdot \mathpzc{C}_r$ and compute matrices $U_d^{r\perp}$
so that $[U_d^r \quad U_d^{r \perp}]$ is orthogonal for each $d$. The vectors $\hat{\vb{u}}_{rj}^d$ are scaled versions of $\vb{e}_j^{(l_d^r)}$ as in \cref{prop:GTuckerOrthBasis}. That is, they are defined such that the transpose of $j$th right singular vector of the $d$th flattening of the $r$th core $(\mathpzc{C}_r)_{(d)}$ is $(\hat{\vb{u}}_j^d)^T (\mathpzc{C}_r)_{(d)}$.
This gives the basis
\[\mathscr{B}_{\mathpzc{G}_r} :=
\left\{
    (U_1^r, \dots, U_D^r) \cdot \dot{\mathpzc{C}_r}
\right\}
\cup
\left\{
    (U_1^r, \dots, U_{d-1}^r, U_d^{r\perp} \vb{e}_i (\hat{\vb{u}}_{rj}^d)^T, U_{d+1}^r, \dots, U_D^r) \cdot \mathpzc{C}_r
\right\}
\]
with $\dot{\mathpzc{C}} \in \mathscr{B}_{\mathpzc{C}}$, $d=1,\dots,D$, $i = 1,\dots,m_d - l_d$ and $j = 1,\dots,l_d^r$.

For $T_{\mathpzc{A}_r}\mathcal{M}_{r}^{n_1,\dots,n_D}$, we can use a basis of the same form, constructed as follows:
We form $Q_d^\perp$ so that the columns of $[Q_d \quad Q_d^\perp]$ are an orthonormal basis  of $\mathbb{R}^{n_d}$. Then define
\[
(Q_d U_d^r)^\perp := \left[ Q_d U_d^{r\perp} \quad Q_d^\perp \right]
\in \mathbb{R}^{n_d \times (n_d - l_d)}
.\]
The columns of this matrix are a basis for the orthogonal complement of the column space of $Q_d U_d^r$. A basis $\mathscr{B}_{\mathpzc{A}_r}$ of $T_{\mathpzc{A}_r} \mathcal{M}_{r}^{n_1,\dots,n_D}$ is obtained by applying \cref{eq:GTuckerOrthBasis} where $Q_d U_d^r$ fulfills the role of $U_d^r$ and $(Q_d U_d^r)^\perp$ fulfills that of $U_d^\perp$.

By rearranging the order of the basis vectors and factoring out all $Q_d$ and $Q_d^\perp$, we get a partition of this basis:
\begin{equation}\label{B_A_r}
\mathscr{B}_{\mathpzc{A}_r} =
(Q_1 \otimes \dots \otimes Q_D) (\mathscr{B}_{\mathpzc{G}_r})
\cup \mathscr{B}_{1}^{r,\perp} \cup \dots \cup \mathscr{B}_{D}^{r,\perp}
,
\end{equation}
where
\[
\mathscr{B}_{d}^{r,\perp} =
\left\{
\left(Q_d^\perp \otimes_d \bigotimes_{d' \ne d} Q_{d'} \right)
    \left(
    (\vb{e}_i^{(n_d - m_d)} (\hat{\vb{u}}_{rj}^d)^T) \otimes_d
    \bigotimes_{d' \ne d} U_d^r
    \right) \mathpzc{C}_r
\right\}_{i,j=1}^{n_d - m_d, l^r_d}
.\]
By construction of $Q_d^\perp$, the subspaces that for a fixed $r$ are spanned by the $D+1$ bases in \cref{B_A_r} are pairwise orthogonal. Therefore, collecting $\mathscr{B}_{\mathpzc{A}_r}$ for all $r$ gives a Terracini matrix of $\mathpzc{A}$, which splits into $D+1$ pairwise orthogonal blocks. Up to a permutation of the columns,
\begin{equation}
    \label{eq:TerraciniAsplit}
    T_{\mathpzc{A}_1,\dots,\mathpzc{A}_R} = \begin{bmatrix}
        (Q_1 \otimes \dots \otimes Q_D)T_{\mathpzc{G}_1,\dots,\mathpzc{G}_R}
        &
        T_1^\perp & \dots & T_D^\perp
    \end{bmatrix},
\end{equation}
where the columns of each $T_d^\perp$ are $\mathscr{B}_d^{1,\perp} \cup \dots \cup \mathscr{B}_d^{R,\perp}$. Explicitly,
\[
T_d^\perp =
\left(Q_d^\perp \otimes_d \bigotimes_{d' \ne d} Q_{d'} \right)
\left[
    \vb{e}_i^{(n_d - m_d)} \otimes_d
    \left( (\hat{\vb{u}}_{rj}^d)^T \otimes_d
    \bigotimes_{d' \ne d} U_d^r
    \right) \mathpzc{C}_r
\right]_{r,i,j=1}^{R, n_d - m_d, l^r_d}.
\]
Because the blocks are pairwise orthogonal, the singular values of $T_{\mathpzc{A}_1,\dots,\mathpzc{A}_R}$ are the union of those of $T_{\mathpzc{G}_1,\dots,\mathpzc{G}_R}$ and those of each $T_d^\perp$ separately.

The factor $\left(Q_d^\perp \otimes_d \bigotimes_{d' \ne d} Q_{d'} \right)$ is orthogonal and thus can be omitted for the purpose of computing singular values. By the definition of the Kronecker product, the columns of the remaining factor can be permuted to
\[
\tilde{T}_d^\perp
:=
\left[
    \mathds{1}_{n_d - m_d} \otimes_d
    \left(
    (\hat{\vb{u}}_{rj}^d)^T \otimes_d
    \bigotimes_{d' \ne d} U_d^r
    \right) \mathpzc{C}_r
\right]_{r,j=1}^{R, l^r_d}
\]
hence, its singular values are just those of $T_d^\perp$. We will show that this is effectively a submatrix of $T_{\mathpzc{G}_1,\dots,\mathpzc{G}_R}$ so that the desired result follows from the interlacing property of singular values \cite[Theorem 8.1.7]{Golub2013}.

For any $r$ and $d$, take all tangent vectors at $\mathpzc{G}_r$ in the set $\mathscr{V}_r^d \cup \mathscr{W}_r^d$ with
\begin{align*}
    \mathscr{V}_r^d &:=
    \left\{
    \left(
        \left( U_d^r \vb{e}_i^{(l_d^r)} (\hat{\vb{u}}_{rj}^d)^T \right) \otimes_d
        \bigotimes_{d' \ne d} U_d^r
    \right) \mathpzc{C}_r
    \right\}_{i,j=1}^{l^r_d, l_d^r} \quad\text{and}\\
    \mathscr{W}_r^d &:=
    \left\{
    \left(
        \left( U_d^{r \perp} \vb{e}_i^{(m_d - l_d^r)} (\hat{\vb{u}}_{rj}^d)^T \right) \otimes_d
        \bigotimes_{d' \ne d} U_d^r
    \right) \mathpzc{C}_r
    \right\}_{i,j=1}^{m_d - l^r_d, l_d^r}
.\end{align*}
In the proofs of \cref{prop:quotientMfd,prop:GTuckerIsMfd}, we showed that all vectors in the same form as $\mathscr{V}_r^d$ are tangent to $\mathcal{M}_{r}^{m_1,\dots,m_D}$.
$\mathscr{W}_r^d$ is just a subset of $\mathscr{B}_{\mathpzc{G}_r}$, the basis we used for~$\mathpzc{G}_r$.
By construction of $U_d^{r\perp}$, the spaces $\mathscr{V}_r^d$ and $\mathscr{W}_r^d$ are orthogonal. The inner products between the elements of $\mathscr{V}_r^d$ (respectively, $\mathscr{W}_r^d$) are of the same form as \cref{eq:orthBasisInnerProducts}. Hence, they are also zero.
By collecting $\mathscr{V}_r^d \cup \mathscr{W}_r^d$ for all $r$, we get a subset of the columns of~$T_{\mathpzc{G}_1,\dots,\mathpzc{G}_R}$:
\[
    \tilde{T}_d^{\mathrm{part}} := \left[
        [U_d^r \quad U_d^{r \perp}] \otimes_d
        \left(
        (\hat{\vb{u}}_{rj}^d)^T \otimes_d
        \bigotimes_{d' \ne d} U_d^r
        \right) \mathpzc{C}_r
    \right]_{r,j=1}^{R, l_r^d},
\]
which has the same singular values as $\tilde{T}_d^\perp$ by \cref{lemma:KroneckerSvdIdentity}. Hence, the singular values of $T_d^\perp$ are interlaced between those of $T_{\mathpzc{G}_1,\dots,\mathpzc{G}_R}$.
By reminding ourselves that \cref{eq:TerraciniAsplit} is a decomposition into pairwise orthogonal blocks, we can see that $T_{\mathpzc{A}_1,\dots,\mathpzc{A}_R}$ and $T_{\mathpzc{G}_1,\dots,\mathpzc{G}_R}$ must have the same extreme singular values.
\end{proof}

\section{Computing the condition number}\label{sec_fastcondition}
One computational advantage of \cref{thm:SBTDcondInvariance} is that for any SBTD that is computed using the compress-decompose-expand strategy from \cref{sec:basicDefs}, the condition number can be computed at a low extra cost right after the decompose phase. That is, it is not necessary to compute the expanded decomposition in order to know its condition number. Furthermore, if a subspace-constrained SBTD $\mathpzc{A}_1 + \dots + \mathpzc{A}_R$ is given, it can be compressed prior to computing its condition number. This gives \cref{alg:CondAfterCompression}.

\begin{algorithm}
    \caption{Computation of $\kappa^{\mathrm{SBTD}}(\mathpzc{A}_1,\dots,\mathpzc{A}_R)$ with $\mathpzc{A}_r = (U_1^r,\dots,U_D^r) \cdot \mathpzc{C}_r$.}
    \label{alg:CondAfterCompression}
\begin{algorithmic}
    \FOR{$d = 1,\dots,D$}{
        \STATE Compute a QR decomposition $Q_d R_d = [U_d^1,\dots,U_d^r]$.
    }\ENDFOR
    \FOR{$r = 1,\dots,R$}{
        \STATE $\mathpzc{G}_r \gets (Q_1^T,\dots,Q_D^T) \cdot \mathpzc{A}_r$
    }\ENDFOR
    \STATE Compute $\kappa^{\mathrm{SBTD}}(\mathpzc{G}_1,\dots,\mathpzc{G}_R)$ using the algorithm from \cref{sec:algorithm}.
\end{algorithmic}
\end{algorithm}

This algorithm was applied to the numerical example mentioned in the introduction. Its computational complexity compares to that of the naive approach as follows.

\begin{proposition}
    \label{prop:condComplexity}
    Let $\mathpzc{A} = \mathpzc{A}_1 + \dots + \mathpzc{A}_R \in \mathbb{R}^{n \times \dots \times n}$ be a subspace-constrained SBTD with core structures $\mathcal{M}_r$, where each $\mathcal{M}_r$ is an open submanifold of $\mathbb{R}^{l \times \dots \times l}$. Assume that the summands $\mathpzc{A}_r$ are given in HOSVD form. Assume that computing the QR and singular value decomposition of an $m \times n$-matrix with $m \ge n$ both take $O(mn^2)$ arithmetic operations.
    The number of arithmetic operations involved in applying \cref{eq:computeJoinCond} directly and applying \cref{alg:CondAfterCompression} is
    \[
    O(n^D R^2 l^{2D} + n^{D}R^2 D^2 l^2 (n-l)^2)
    \quad\text{and}\quad
    O( DnR^2l^2 + R^{D+2}l^{3D} + R^{D+4} l^{D+4} D^2),
    \]
    respectively.
\end{proposition}
\begin{proof}

    First, we apply \cref{eq:GTuckerOrthBasis} directly to $\mathpzc{A}_r = (U_1^r,\dots,U_D^r) \cdot \mathpzc{C}_r$.
    Computing the complement $U_d^{r\perp}$ of $U_d^r$ is negligible.
    The basis vectors in \cref{eq:GTuckerOrthBasis} with indices $i,j,d$ can be computed as tensors whose $d$th unfolding is
    $
    U_d^{r\perp} \vb{e}_i (U_d \vb{e}_j)^T (\mathpzc{A}_r)_{(d)}
    $, which takes $O(ln^D)$ operations per basis vector. Computing $U_1^r \otimes \dots \otimes U_D^r$ takes $O(n^D l^D)$ time.
    This gives a time of $O(Rn^Dl^D + Rl^2(n-l)n^D)$ to construct the full Terracini matrix, whose dimensions are $n^D \times p$ where $p = R(l^D + Dl(n-l))$. Computing its singular values requires $O(n^Dp^2) = O(n^D R^2 l^{2D} + n^{D}R^2 D^2 l^2 (n-l)^2)$ operations \cite{Golub2013}.

    Next, we consider \cref{alg:CondAfterCompression}. The matrices $[U_d^1,\dots,U_d^r]$ have  $m := Rl$ columns and $n$ rows, which gives a complexity of $O(nR^2l^2)$ for each each QR decomposition \cite{Golub2013}. Converting each $\mathpzc{G}_r$ to HOSVD form takes $O(Dm^{D+1})$ time \cite{Vannieuwenhoven2012}.
    Constructing the Terracini matrix is negligible compared to computing its singular values, as before. In this case, the Terracini matrix has dimensions $m^D \times q$ where $q = R(l^D + Dl(m-l)) = O(R l^D + R^2 Dl^2)$. The computation of the singular values requires $O(m^D q^2) = O(R^D l^D q^2) = O(R^{D+2}l^{3D} + R^{D+4} l^{D+4} D^2)$ operations.
\end{proof}

If $n$ is significantly larger than $Rl$ in this proposition, the complexity is approximated by $O(n^{D+2} R^2D^2l^2)$ and $O(Rln^D)$, respectively, which shows the superiority of \cref{alg:CondAfterCompression}. On the other hand, if $n \le Rl$, the algorithm does not compress the decomposition and merely adds overhead.

\section{Numerical experiments}
\label{sec:experiments}
\begin{figure}[p]
    \centering
    \includegraphics[width=0.83\textwidth]{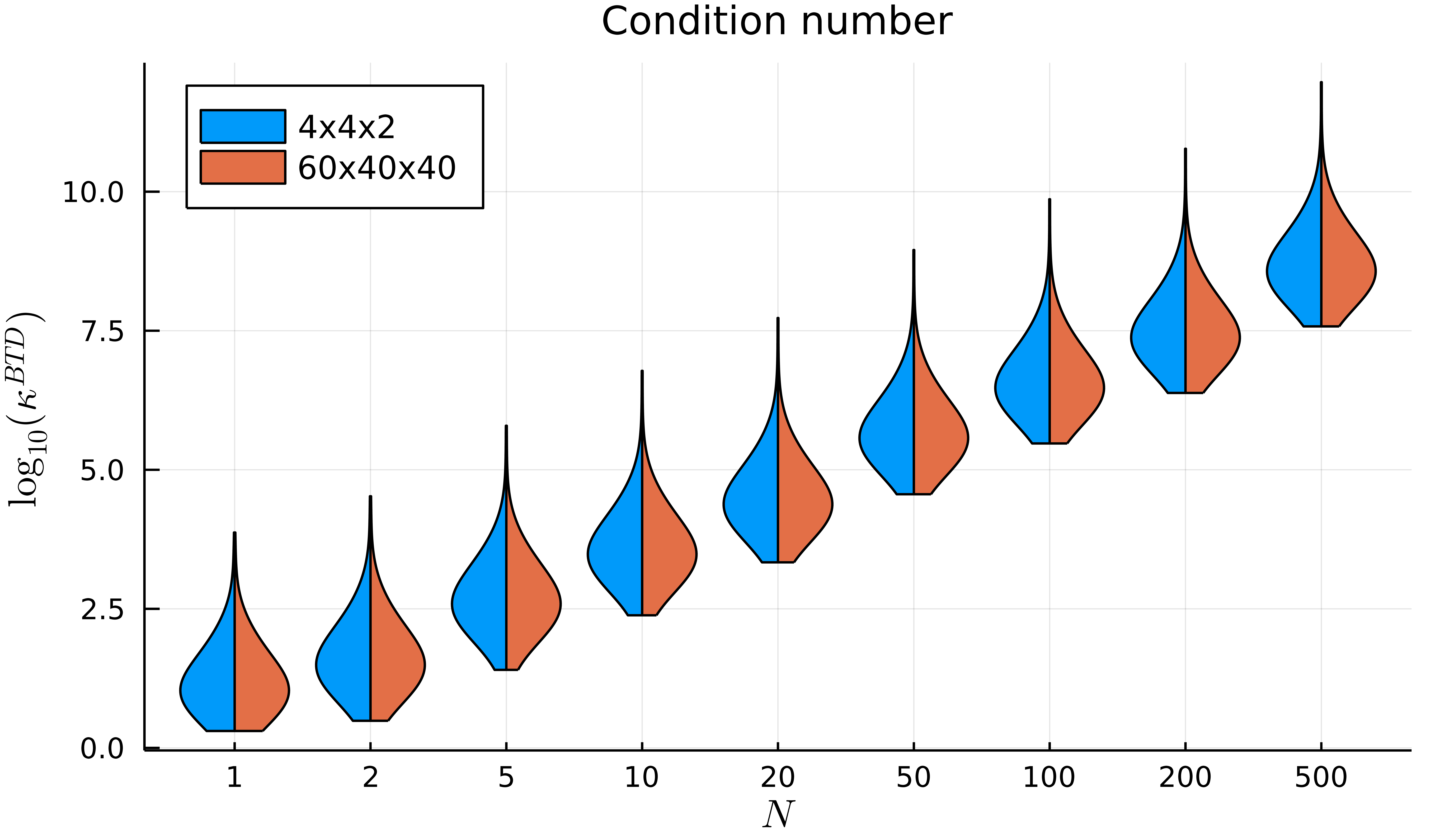}
    \caption{Condition number of the BTD of $\mathpzc{G}_N \in \mathbb{R}^{4 \times 4 \times 2}$ and that of $\mathpzc{A}_N \in \mathbb{R}^{60 \times 40 \times 40}$ from the experiments in \cref{sec:experiments}}
    \label{fig:violin_cond}
\end{figure}
\begin{figure}[p]
    \centering
    \includegraphics[width=0.83\textwidth]{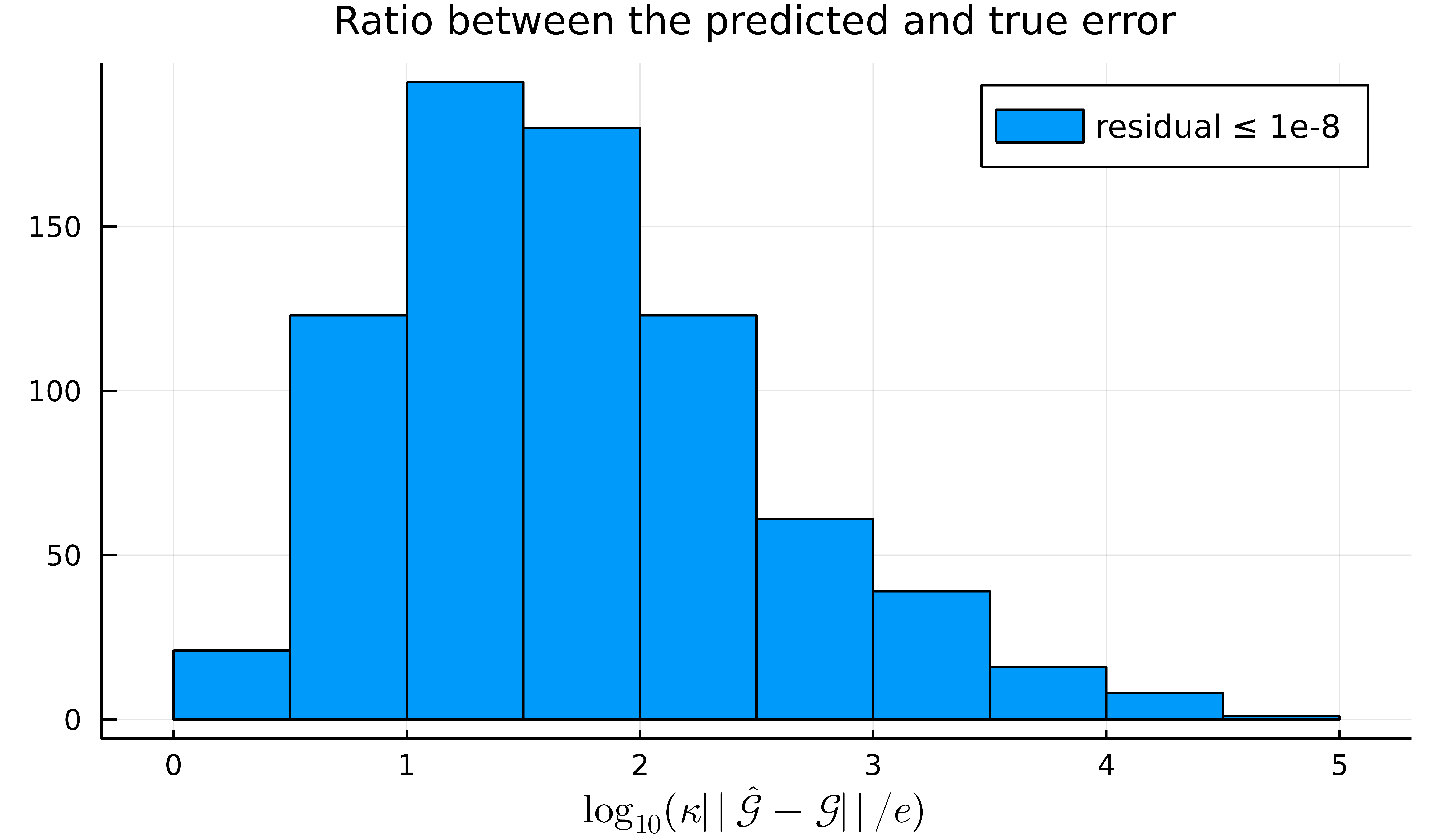}
    \caption{Ratio between the estimated forward error based on \cref{eq:fstOrderErrBound} and the true forward error for $\mathpzc{G}_N$ in the experiments in \cref{sec:experiments}. Only cases with a residual $\norm{\hat{\mathpzc{G}} - \mathpzc{G}} \le 10^{-8}$ were considered.}
    \label{fig:hist_forward_backward}
\end{figure}
\begin{figure}[p]
    \centering
    \includegraphics[width=0.83\textwidth]{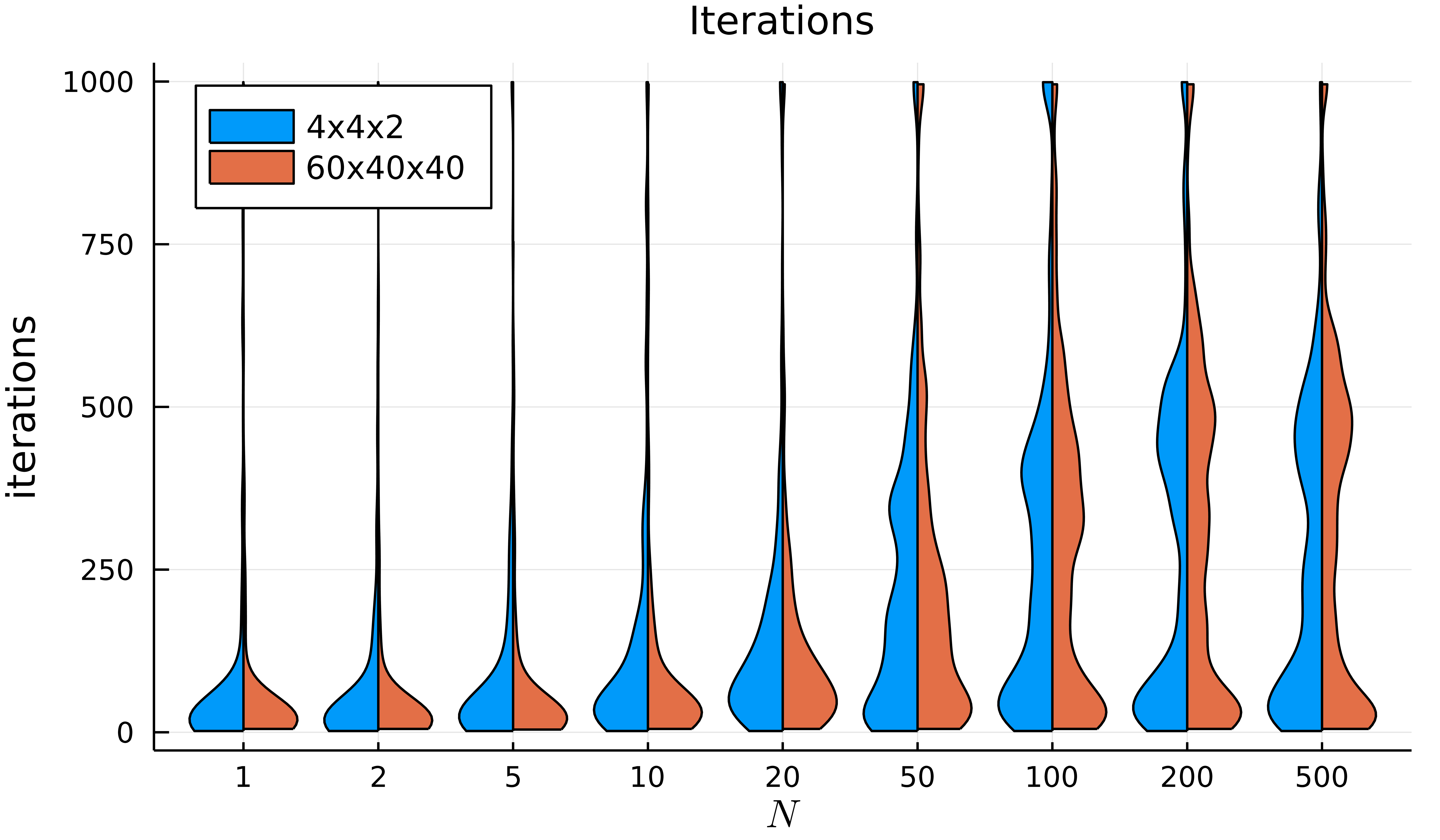}
    \caption{Number of iterations of \texttt{btd\_nls} applied to $\mathpzc{G}_N \in \mathbb{R}^{4 \times 4 \times 2}$ and $\mathpzc{A}_N \in \mathbb{R}^{60 \times 40 \times 40}$ from the experiments in \cref{sec:experiments}.}
    \label{fig:violin_iterations}
\end{figure}

We present a few numerical experiments illustrating the main result, \cref{thm:SBTDcondInvariance}, with a sequence of ill-conditioned block term decompositions.
All numerical computations were performed on an Intel Xeon CPU E5-2697 v3 running on 16 cores and 126GB memory. The tensor decompositions were computed in MATLAB R2018b with Tensorlab 3.0 \cite{Vervliet2017} and the other computations were performed in Julia v1.6.\ \cite{BEKV}.

De Silva and Lim \cite{DeSilva2008} give an explicit parametrisation of a general curve of rank-$2$ tensors $\mathpzc{X}_N$ that converges to a rank-$3$ tensor as $N \rightarrow \infty$. In such cases, the condition number diverges to infinity \cite{Breiding2018a}. Given the vectors $\vb{x}_d$ and $\vb{y}_d$ for $d=1,2,3$,
the sequence $\{\mathpzc{X}_N\}_{N=1}^\infty$ is given by
\[
N \bigotimes_{d=1}^3 \left(\vb{x}_d + \frac{1}{N} \vb{y}_d\right) - N \bigotimes_{d=1}^3 \vb{x}_d
=
\vb{y}_1 \otimes \vb{x}_2 \otimes \vb{x}_3
+
\vb{x}_1 \otimes \vb{y}_2 \otimes \vb{x}_3
+
\vb{x}_1 \otimes \vb{x}_2 \otimes \vb{y}_3
+
\mathcal{O}\left(\frac{1}{N}\right)
.\]
This example can easily be generalised to block term decompositions. Take any third-order core tensor $\mathpzc{C}$ of full multilinear rank and any two sets of full-rank matrices $\{A_d\}_{d=1}^3$ and $\{B_d\}_{d=1}^3$. Then set
\begin{equation}
    \label{eq:modelIllCondBTD}
\mathpzc{G}_N := \left(
    N \bigotimes_{d=1}^3 \left(B_d + \frac{1}{N} A_d\right) - N \bigotimes_{d=1}^3 B_d
\right) \mathpzc{C}
.\end{equation}
Both blocks have the same multilinear rank assuming $B_d$ and $B_d + \frac{1}{N}A_d$ have full rank. Similarly to $\mathpzc{X}_N$, we can see that $\mathpzc{G}_N$ equals a three-term BTD independent of~$N$, plus $o(N^{-1})$ terms.
Its condition number diverges as $N \rightarrow \infty$ by a special case of \cite[Theorem 1.4]{Breiding2018a}.


We generated tensors of this model where the all the core tensor $\mathpzc{C} \in \mathbb{R}^{2 \times 2 \times 1}$ and the matrices $A_1, A_2 \in \mathbb{R}^{4 \times 2}, A_3 \in \mathbb{R}^{2 \times 1}$ all have standard normally distributed entries and $B_d$ is the Q-factor of the QR decomposition of a matrix with standard normal entries. For several values of $N$, we generated 2000 tensors of model \cref{eq:modelIllCondBTD}. For each of these we generated an ``inflated'' version $\mathpzc{A}_N = (Q_1, Q_2, Q_3) \cdot \mathpzc{G}_N$ for some $Q_1,Q_2,Q_3$ with orthonormal columns. The dimensions of the tensors are $\mathpzc{G}_N \in \mathbb{R}^{4 \times 4 \times 2}$ and $\mathpzc{A}_N \in \mathbb{R}^{60 \times 40 \times 40}$.

We used Tensorlab's Gauss--Newton method \texttt{ll1\_nls} \cite{Vervliet2017} to compute a two-term $(2,2,1)$-BTD of both the (sequences of) tensors $\mathpzc{A}_N$ and $\mathpzc{G}_N$ independently. Since $\mathpzc{A}_N$ has a subspace-constrained BTD with core tensor $\mathpzc{G}_N$, by \cref{thm:SBTDcondInvariance} their condition numbers are the same. 
Some built-in optimisations were disabled, namely automatic Tucker compression and the use of the iterative solver to solve the linear system to compute the quasi--Newton update direction. This is to ensure the same algorithm is used for both tensors. Since \texttt{ll1\_nls} stops when the backward error reaches a certain threshold, this generates exact decompositions of nearby tensors, which allows us to compare the forward and backward error.

A violin plot of the condition number of both BTDs is shown in \cref{fig:violin_cond}. The condition number does indeed increase with the parameter $N$. Moreover, the distribution of the condition number of the BTD of $\mathpzc{G}_N$ is the same as that of~$\mathpzc{A}_N$.  We did find that the ratio between the \textit{computed} condition numbers $\hat{\kappa}$ deviated slightly from one in the more ill-conditioned cases.
The most extreme case was $\hat{\kappa}(\mathpzc{A}_1,\dots,\mathpzc{A}_R) \approx (1 - 2\cdot 10^{-5}) \hat{\kappa}(\mathpzc{G}_1,\dots,\mathpzc{G}_R)$ where $\kappa > 10^{12}$. We attribute this to numerical roundoff. These results thus provide a numerical verification of \cref{thm:SBTDcondInvariance}.

A major application of the condition number is to estimate the forward error.  For a true decomposition $ \mathpzc{G} = \sum_{r=1}^R \mathpzc{G}_r$ and a computed decomposition $\hat{\mathpzc{G}} = \sum_{r=1}^R \hat{\mathpzc{G}}_r$, the forward error is measured as
\[
e = \min_{\pi \in \mathscr{S}_R} \sqrt{ \sum_{r=1}^r \norm{\mathpzc{G}_r - \hat{\mathpzc{G}}_{\pi(r)}}^2},
\]
where $\mathscr{S}_R$ is the symmetric group of $R$ elements. By \cref{eq:fstOrderErrBound} we can estimate that $e \lesssim \kappa^{\mathrm{BTD}} \norm{\mathpzc{G} - \hat{\mathpzc{G}}}$
as long as the residual $\norm{\mathpzc{G} - \hat{\mathpzc{G}}}$ is not too large. \Cref{fig:hist_forward_backward} shows that this bound tends to hold when the residual is at most $10^{-8}$.

Finally, condition numbers tend to be related to convergence rate. Because the condition number is equal for $\mathpzc{G}_N$ and $\mathpzc{A}_N$, one could expect the convergence behaviour to be the same. This is reflected in the number of iterations, see \cref{fig:violin_iterations}. It also shows that convergence gets slower as the condition number increases.

The cost per iteration is expected to be a function of only the dimensions of the tensor and the block terms, as only direct linear algebra routines are used to compute the iteration steps \cite{Sorber2013a}. By using the compressed tensor $\mathpzc{G}_N$ instead of $\mathpzc{A}_N$, the geometric mean of the speedup per iteration was 9.5. In \cite{Bro1998}, speedup factors of up to 40 were observed for the ALS algorithm applied to tensors used in chemometrics. Note that this is the speedup of computing the decomposition. For the sugar data set of \cite{Bro1998}, the computation of the \textit{condition number}, as mentioned in the introduction, was sped up by a factor of $15\,000$ by first Tucker compressing the tensor from size $265 \times 371 \times 7$ to $3 \times 3 \times 3$.

\bibliographystyle{siam}
\bibliography{library}

\end{document}